\documentclass[a4paper,11pt]{article}
\usepackage[latin1]{inputenc}
\usepackage[english]{babel}
\usepackage{amsmath}
\usepackage{amsfonts}
\usepackage{amssymb}
\usepackage{epsfig}
\usepackage{amsopn}
\usepackage{amsthm}
\usepackage{color}
\usepackage{graphicx}
\usepackage{subfigure}
\usepackage{enumerate}
\usepackage{tikz}
\usetikzlibrary{intersections}
\newtheorem{theorem}{Theorem}[section]

\newtheorem{lemma}[theorem]{Lemma}
\newtheorem{proposition}[theorem]{Proposition}
\newtheorem{definition}[theorem]{Definition}

\newtheorem*{theorem*}{Theorem}
\newtheorem*{lemma*}{Lemma}
\newtheorem*{remark*}{Remark}
\newtheorem*{definition*}{Definition}
\newtheorem*{proposition*}{Proposition}
\newtheorem*{corollary*}{Corollary}
\numberwithin{equation}{section}
\newcommand{\real}{\mathbb{R}}

\let\ced=\c         

\def\x{\xi}

\def\qed{\,\unskip\kern 6pt \penalty 500
\raise -2pt\hbox{\vrule \vbox to8pt{\hrule width 6pt
\vfill\hrule}\vrule}\par}
\definecolor{darkblue}{rgb}{0.05, .05, .65}
\definecolor{darkgreen}{rgb}{0.1, .65, .1}
\definecolor{darkred}{rgb}{0.8,0,0}
\newcommand{\beqn}{\begin{equation}}
\newcommand{\eeqn}{\end{equation}}
\newcommand{\bear}{\begin{eqnarray}}
\newcommand{\eear}{\end{eqnarray}}
\newcommand{\bean}{\begin{eqnarray*}}
\newcommand{\eean}{\end{eqnarray*}}
\begin{document}

\title{\huge \bf Boundedness and evolution rates for a quasilinear reaction-diffusion equation}

\author{
\Large Razvan Gabriel Iagar\,\footnote{Departamento de Matem\'{a}tica
Aplicada, Ciencia e Ingenieria de Materiales y Tecnologia
Electr\'onica, Universidad Rey Juan Carlos, M\'{o}stoles,
28933, Madrid, Spain, \textit{e-mail:} razvan.iagar@urjc.es},\\
[4pt] \Large Marta Latorre\,\footnote{Departamento de Matem\'{a}tica
Aplicada, Ciencia e Ingenieria de Materiales y Tecnologia
Electr\'onica, Universidad Rey Juan Carlos, M\'{o}stoles,
28933, Madrid, Spain, \textit{e-mail:} marta.latorre@urjc.es},
\\[4pt] \Large Ariel S\'{a}nchez\footnote{Departamento de Matem\'{a}tica
Aplicada, Ciencia e Ingenieria de Materiales y Tecnologia
Electr\'onica, Universidad Rey Juan Carlos, M\'{o}stoles,
28933, Madrid, Spain, \textit{e-mail:} ariel.sanchez@urjc.es}\\
[4pt] }
\date{}
\maketitle

\begin{abstract}
We consider the following quasilinear reaction-diffusion equation
$$
\partial_tu=\Delta u^m+(1+|x|)^{\sigma}u^p, \quad (x,t)\in\real^N\times(0,\infty),
$$
in dimension $N\geq3$ and in the range of exponents $1<p<m$ and $-\infty<\sigma<-2$. We prove that, for initial conditions $u_0$ satisfying
$$
u_0\geq0, \quad u_0\not\equiv0, \quad \lim\limits_{|x|\to\infty}|x|^{-(\sigma+2)/(m-p)}u_0(x)=0,
$$
the solution $u$ to the corresponding Cauchy problem remains uniformly bounded from above and below:
$$
C_1\leq \|u(t)\|_{\infty}\leq C_2, \quad t\in(0,\infty),
$$
for some positive constants $C_1$ and $C_2$. Under suitable conditions on $p$, we also establish the rate of expansion of the upper limit $R(t)$ of the positivity set for compactly supported data, that is,
$$
At^{\beta}\leq R(t)\leq Bt^{\beta}, \quad \beta=-\frac{m-p}{\sigma(m-1)+2(p-1)},
$$
and a \emph{different time scale in outer sets}, that is
$$
D_1t^{-\alpha}\leq u(x,t)\leq D_2t^{-\alpha}, \quad \alpha=\frac{\sigma+2}{\sigma(m-1)+2(p-1)}, \quad {\rm if} \ |x|\geq Ct^{\beta}.
$$
The boundedness is in striking contrast with the property of grow-up as $t\to\infty$ established in previous works by the authors for $\sigma>-2$, illustrating the character of threshold of the exponent $\sigma=-2$.
\end{abstract}

\

\noindent {\bf Mathematics Subject Classification 2020:} 35B33, 35B40, 35K57, 35K59.

\smallskip

\noindent {\bf Keywords and phrases:} reaction-diffusion equations, global solutions, large time behavior, uniform boundedness, rates.

\section{Introduction}

This paper addresses the question of understanding the dynamics of the Cauchy problem associated to the quasilinear reaction-diffusion equation
\begin{equation}\label{eq1}
\partial_tu=\Delta u^m+(1+|x|)^{\sigma}u^p, \quad (x,t)\in\real^N\times(0,\infty),
\end{equation}
with
\begin{equation}\label{ic}
u(x,0)=u_0(x)\geq0, \quad x\in\real^N,
\end{equation}
in the range of exponents and parameters
\begin{equation}\label{range.exp}
m>1, \quad p\in(1,m), \quad \sigma\in(-\infty,-2), \quad N\geq3.
\end{equation}
The initial condition is assumed to satisfy, in the more general case, the conditions
\begin{equation}\label{icond}
u_0\in C(\real^N), \quad u_0\geq0, \quad u_0\not\equiv0, \quad \lim\limits_{|x|\to\infty}|x|^{-(\sigma+2)/(m-p)}u_0(x)=0.
\end{equation}
The subclass of compactly supported initial conditions also plays an important role in the paper. The restriction on the dimension $N\geq3$ is due to the fact that precedents in the literature have shown that dimensions $N=1$ and $N=2$ are very different with respect to the properties of Eq. \eqref{eq1} and related equations.

The aim of this study is, besides advancing in the understanding of the behavior of solutions to Eq. \eqref{eq1}, to illustrate the contrast between the grow-up regime as $t\to\infty$ established in the range $\sigma>-2$ by the authors in their previous work \cite{ILS25} and the uniform boundedness (also with respect to time) of the solutions in the range of exponents \eqref{range.exp}.

The basic theory of the Cauchy problem \eqref{eq1}-\eqref{ic} can be found in \cite{AdB91}, including optimal conditions for its well-posedness. In particular, it is easy to observe that the conditions \eqref{icond} are included in the ones ensuring well-posedness, as will be discussed together with our main results.

Eq. \eqref{eq1} is a reaction-diffusion equation whose reaction term features a weight depending on the spatial variable. The dynamical properties of such equations depend on the competition between the two terms in the right-hand side: on the one hand, a diffusion term that spreads the total mass (or $L^1$ norm, if the solution is integrable) from regions near the origin to outer regions and, on the other hand, a source term (or reaction) that produces at every $t>0$ an increase of the $L^1$ norm of the solution. It is well-established in literature that the presence of this source term might produce the formation of singularities or regions where the solution becomes unbounded, a mathematical phenomenon known as finite time blow-up if it occurs as $t\to T$ for some $T\in(0,\infty)$, or grow-up if the solution is unbounded as $t\to\infty$ (but it remains bounded at every $t\in(0,\infty)$).

Indeed, for the standard (autonomous) reaction-diffusion equation
$$
u_t=\Delta u^m+u^p,
$$
with $1<p<m$, which corresponds to taking $\sigma=0$ in \eqref{eq1}, it is known that any nontrivial solution blows up in a finite time (see for example \cite[Section 3.1, Chapter IV]{S4}). The authors extended in \cite{ILS24} this property of universal finite time blow-up to the following weighted reaction-diffusion equation
\begin{equation}\label{eq0}
u_t=\Delta u^m+|x|^{\sigma}u^p
\end{equation}
in the range of exponents $1<p<m$ and $\sigma>0$, featuring a weight which is unbounded as $|x|\to\infty$. More recently, the blow-up rates for \eqref{eq0} have also been established in \cite{FIS26} when $\sigma>0$ and correspond to the self-similarity exponents of the equation.

The dynamical properties of Eq. \eqref{eq0} (and also Eq. \eqref{eq1}) strongly depart from those described in the previous paragraph if $\sigma<0$ is considered. Indeed, it has been shown in \cite{IL26} (the exponent being actually discovered in \cite{IMS23}) that, if $\max\{-2,-N\}<\sigma<0$, there exists a critical reaction exponent
$$
p_G:=1-\frac{\sigma(m-1)}{2}\in(1,m),
$$
such that

$\bullet$ if $1<p\leq p_G$, any nontrivial solution to Eq. \eqref{eq0} with a suitable initial condition (for example compactly supported) presents grow-up as $t\to\infty$, with a power-like grow-up rate for $p\in(1,p_G)$ and an exponential grow-up rate if $p=p_G$.

$\bullet$ if $p>p_G$, any nontrivial solution to Eq. \eqref{eq0} blows up in a finite time, similarly as in the case when $\sigma\geq0$, with the self-similar blow-up rate.

Following \cite{AdB91}, the authors have considered, in the same range of exponents
$$
\max\{-2,-N\}<\sigma<0, \quad p\in(1,p_G),
$$
Eq. \eqref{eq1} instead of Eq. \eqref{eq0} in their recent work \cite{ILS25}. The main difference between the two equations is that, while the weight $|x|^{\sigma}$ is singular at $x=0$ when $\sigma<0$, the new weight $(1+|x|)^{\sigma}$ is bounded in $\real^N$ and the well-posedness results in \cite{AdB91} apply to it for any value of $\sigma<0$. Thus, we have addressed in \cite{ILS25} the question of the large time behavior of solutions and proved that, as $t\to\infty$, an interesting asymptotic simplification takes place and the solutions to the Cauchy problem \eqref{eq1}-\eqref{ic} converge to the same asymptotic profile as the solutions to the analogous Cauchy problem for \eqref{eq0}, provided that the initial conditions are compactly supported.

However, the large time behavior of solutions to the Cauchy problem \eqref{eq1}-\eqref{ic} seems to strongly depart from that of Eq. \eqref{eq0} if $\sigma\leq-2$, as it is very likely (although not established in a rigorous manner up to our knowledge) that Eq. \eqref{eq0} has no bounded solutions at all for $\sigma\leq-2$. In particular, it has been shown in \cite{IS23} that, at the level of self-similar solutions to Eq. \eqref{eq0}, for $\sigma=-2$ all such solutions are unbounded and in particular present a vertical asymptote at $x=0$ (and, as we shall see in this paper, this result extends to the range $\sigma<-2$). Thus, a natural question is raised in view of the analogous large time behavior between Eqs. \eqref{eq1} and \eqref{eq0} established in \cite{ILS25}: is this similarity true also when $\sigma<-2$ or does it depart from it? The answer is that the solutions to Eq. \eqref{eq1} are uniformly bounded for $\sigma<-2$ and thus strongly depart from those of Eq. \eqref{eq0}, as well as from the properties shared by the same Eq. \eqref{eq1} when $\sigma>-2$, and the goal of this paper is to prove this fact.

In order to close this discussion of precedents, we recall that a number of works \cite{FdPV06, FdP18, FdP20, FdP22, KWZ11, Liang12} addressed the question of large time behavior (in the sense of establishing ranges for blow-up, grow-up or boundedness) for the Cauchy problem to the following equation
\begin{equation}\label{eq2}
u_t=\Delta u^m+a(x)u^p,
\end{equation}
where $a(x)$ is a compactly supported function (taking as model a characteristic function of an interval or a ball). It has been noticed therein that the dimensions $N=1$ and $N=2$ are special for Eq. \eqref{eq2}, since, even in the range $1<p<m$, grow-up as $t\to\infty$ occurs in dimension $N=2$ (see \cite{FdP18}), while in dimension $N=1$ even finite time blow-up may occur, and the critical exponent $p_0=(m+1)/2$ serves as a threshold between the range where grow-up and finite time blow-up occur (see \cite{FdPV06, FdP18}). Since solutions to Eq. \eqref{eq2} play, after a suitable rescaling, the role of subsolutions to Eq. \eqref{eq1}, it is easy to notice that the latter behavior is also enjoyed by solutions to the Cauchy problem \eqref{eq1}-\eqref{ic}, but establishing rates of grow-up or blow-up is a difficult problem. Thus, we have decided to avoid in this paper these very different low dimensions $N=1$ and $N=2$ and consider them in a further separate work: in these dimensions the results strongly depart from those in dimension $N\geq3$ and new critical exponents have to be considered (at least in dimension $N=1$).

We are thus ready to present our main results.

\medskip

\noindent \textbf{Main results.} Assume that $m$, $p$ and $\sigma$ verify \eqref{range.exp}. Before stating our main results, let us introduce the concept of solution that we employ throughout the paper, following \cite{AdB91}.
\begin{definition}\label{def.weak}
We say that $u:\real^N\times(0,T)\longrightarrow[0,\infty)$ is a weak solution to the Cauchy problem \eqref{eq1}-\eqref{ic} if for every bounded open set $\Omega\subset\real^N$ with smooth boundary the following properties are fulfilled:
\begin{itemize}
  \item \textbf{Regularity:} $u\in C_{\rm loc}(0,T;L^1_{\rm loc}(\real^N))\cap L^{\infty}_{\rm loc}(\Omega\times(0,T))$ and $u^m\in L^2_{\rm loc}(0,T;H^{1}_{\rm loc}(\real^N))$.
  \item \textbf{Weak formulation:} For any function $\eta\in H_0^1(\Omega\times(0,T))$ and for any $t\in(0,T)$ we have
\begin{equation}\label{weaksol}
\begin{split}
\int_{\Omega}u(x,t)\eta(x,t)\,dx&+\int_0^t\int_{\Omega}\left(\nabla u^m\cdot\nabla\eta-u\partial_{t}\eta\right)\,dx\,d\tau\\
&=\int_0^t\int_{\Omega}(1+|x|)^{\sigma}u^p\eta\,dx\,d\tau.
\end{split}
\end{equation}
  \item \textbf{Initial trace:} For any $\eta\in C_0^{\infty}(\real^N)$ we have
$$
\lim\limits_{t\to0}\int_{\real^N}u(x,t)\eta(x)\,dx=\int_{\real^N}u_0(x)\eta(x)\,dx.
$$
\end{itemize}
\end{definition}
In this formulation, and taking into account the range of exponents \eqref{range.exp}, local well-posedness of weak solutions to the Cauchy problem \eqref{eq1}-\eqref{ic} is established in \cite{AdB91} under the assumption that $u_0\in L^1_{\rm loc}(\real^N)$. Indeed, assuming that
\begin{equation*}
[[u_0]]_1:=\sup\limits_{r\geq1}\frac{1}{r^{2/(m-1)}|B(0,r)|}\int_{B(0,r)}|u_0(y)|\,dy<\infty,
\end{equation*}
a fact that is obviously satisfied by any bounded initial condition $u_0$, and taking into account that $p<m$, existence of a weak solution at least on an interval $(0,T)$ is ensured by \cite[Theorem 3.1, p. 371]{AdB91}, while the uniqueness and the comparison principle follow from \cite[Theorem  5.1]{AdB91} and \cite[Proposition 2.2]{Su02}.

Once the well-posedness is settled, let us go to our main results. We assume from now on that $u_0$ satisfies \eqref{icond}, unless something else is specified. In all the following statements, $u$ designs the unique solution to the Cauchy problem \eqref{eq1}-\eqref{ic}.

The first result is a boundedness one and is, in our opinion and by comparison with the range $\sigma>-2$, the most striking one in the paper.
\begin{theorem}\label{th.bound}
Let $m$, $p$, $\sigma$ and $N$ be as in \eqref{range.exp} and assume that $u_0$ satisfies \eqref{icond}. Then there exist two positive constants $0<C_1<C_2<\infty$ such that
\begin{equation}\label{bound}
C_1\leq\|u(t)\|_{\infty}\leq C_2, \quad t\in(0,\infty).
\end{equation}
\end{theorem}
In other words, we prove that the maximum of the solutions to Eq. \eqref{eq1} remains bounded forever between two constants that are independent of time. This contrasts with the unboundedness established for any solution to Eq. \eqref{eq1} in the range $\sigma\in(-2,0)$, where any nontrivial nonnegative solution presents either finite time blow-up or grow-up as $t\to\infty$. Taking into account that, if instead of a weight $(1+|x|)^{\sigma}$ we have a constant in front of the source term, then any nontrivial solution blows up in a finite time as well, we observe that the large time behavior is decisively influenced by the behavior as $|x|\to\infty$ of the weight (and of the initial condition). Thus, a sufficiently rapidly decaying weight such as $(1+|x|)^{\sigma}$ with $\sigma<-2$, together with initial conditions presenting a tail as in \eqref{icond}, ensure uniform boundedness at any time, despite the effect of the source term. As we shall see, the fact that $p<m$ also plays a decisive role in this boundedness, and in fact our expectation is that, if $p>m$, many solutions can still blow up in a finite time, but this range will not be addressed in this work.

For the following results, let us introduce the condition:
\begin{equation}\label{pscond}
\sigma>\sigma_0:=\frac{N-2}{m}-N, \quad 1<p<p_c(\sigma):=\frac{m(N+\sigma)}{N-2}.
\end{equation}
Note that the first inequality in \eqref{pscond} is needed in order to ensure the fact that $p_c(\sigma)>1$. Moreover, if $N\geq3$, then
$$
\frac{N-2}{m}-N+2=-\frac{(N-2)(m-1)}{N}<0,
$$
which implies that $-N<\sigma_0<-2$ and thus there is no contradiction between the conditions \eqref{range.exp} and \eqref{pscond}.

As it is well established, self-similar solutions are essential in the understanding nonlinear diffusion and reaction-diffusion equations and, even if there is no particular self-similar solution to a specific equation, a self-similar behavior is still sometimes observed for general solutions. In our case, introducing the self-similar ansatz
\begin{equation}\label{SSS}
u(x,t)=t^{-\alpha}f(|x|t^{-\beta}),
\end{equation}
in Eq. \eqref{eq0}, one readily finds the exponents
\begin{equation}\label{SSexp}
\alpha:=\frac{\sigma+2}{\sigma(m-1)+2(p-1)}>0, \quad \beta:=-\frac{m-p}{\sigma(m-1)+2(p-1)}>0.
\end{equation}
The positivity of $\alpha$ and the self-similar form \eqref{SSS} might suggest that any self-similar solution to Eq. \eqref{eq0} decays with time. This is actually not true, since there are no bounded self-similar solutions to Eq. \eqref{eq0}. However, the exponents $\alpha$ and $\beta$ play a role in the study of Eq. \eqref{eq1}, as seen in the forthcoming results. We state first a theorem establishing a decay rate in sets lying far from the origin.
\begin{theorem}\label{th.outer}
Let $m$, $p$, $\sigma$ and $N$ be as in \eqref{range.exp} and such that the condition \eqref{pscond} is fulfilled and assume that $u_0$ satisfies \eqref{icond}. Then, for $\delta>0$ sufficiently small, there are positive constants $0<D_1<D_2<\infty$ (depending on $\delta$) and a time $t_0>0$ such that
\begin{equation}\label{decay}
D_1t^{-\alpha}\leq \|u(t)\|_{L^{\infty}(\mathcal{O}_{\delta}(t))}\leq D_2t^{-\alpha}, \quad \mathcal{O}_{\delta}(t):=\{x\in\real^N: |x|\geq\delta t^{\beta}\},
\end{equation}
for any $t>t_0$. In the absence of the condition \eqref{pscond}, the lower bound in \eqref{decay} is still in force.
\end{theorem}
The sets of the form $\mathcal{O}_{\delta}$ are usually known as \emph{outer sets} or \emph{far field sets} and are useful in describing the behavior of solutions to nonlinear diffusion equations. The reason for this is the self-similar structure of such equations: the variables defined by
$$
v(y,\tau)=t^{\alpha}u(x,t), \quad y=xt^{\beta}, \quad \tau=\ln\,t
$$
are called self-similar variables and regularly employed in the development of the theory of nonlinear diffusion. Thus, the decay rate \eqref{decay} is understood as a uniform boundedness of the solutions rescaled in self-similar variables, but in sets (in variables $(y,\tau)$) of the form $|y|\geq\delta>0$, that is, avoiding the origin (otherwise, a contradiction with Theorem \ref{th.bound} would appear).

As an outcome of Theorems \ref{th.bound} and \ref{th.outer}, a very interesting phenomenon is obtained for Eq. \eqref{eq1}: it has a large time behavior with \emph{two different time scales}, a uniform boundedness actually stemming from the behavior on compact sets in variables $(x,t)$ and a self-similar decay with a rate $t^{-\alpha}$ in outer sets limited by the self-similar curves $|x|=\delta t^{\beta}$ for any $\delta>0$. Such a very interesting phenomenon has been also noticed in \cite{KRV10} for the non-homogeneous porous medium equation
$$
\varrho(x)u_t=\Delta u^m, \quad \varrho(x)\sim|x|^{-\gamma} \quad {\rm as} \ |x|\to\infty,
$$
for $\gamma>2$, despite the fact that there is no obvious connection between the previous equation and Eq. \eqref{eq1}.

Restricting ourselves to compactly supported initial conditions, we establish the finite speed of propagation of the support and its expansion rate, which corresponds to the self-similar exponents again.
\begin{theorem}\label{th.supp}
Let $m$, $p$, $\sigma$ and $N$ be as in \eqref{range.exp} and assume that $u_0$ satisfies \eqref{icond} and is compactly supported. Then, $u(t)$ remains compactly supported for any $t\in(0,\infty)$. If, moreover, the condition \eqref{pscond} is fulfilled, denoting by
$$
R(t):=\sup\{|x|: x\in\real^N, u(x,t)>0\}
$$
the maximal length of the positivity set of $u$ at time $t$, then there exist two positive constants $0<c_1<c_2<\infty$ such that
\begin{equation}\label{supp}
c_1t^{\beta}\leq R(t)\leq c_2t^{\beta}, \quad t\in(1,\infty).
\end{equation}
In the absence of the condition \eqref{pscond}, the lower bound in \eqref{supp} is still in force.
\end{theorem}

We complete the results of the paper with the following \emph{improved lower bound}, which significantly strengthens the conclusion of Theorem \ref{th.bound}. For the sake of completeness and consistency with the previous theorems, we state the result together with the corresponding upper bound (which follows immediately from Theorem \ref{th.bound} and its proof). The essential novelty, however, lies in the lower bound.
\begin{theorem}\label{th.lower}
Let $m$, $p$, $\sigma$ and $N$ be as in \eqref{range.exp} and such that the condition \eqref{pscond} is fulfilled. Assume that $u_0$ satisfies \eqref{icond}. Then there exist positive constants $K_1$, $K_2\in(0,\infty)$ with $K_1<K_2$ such that, for any compact set $D\subset\real^N$, there exists a time $t_{D}>0$ satisfying
\begin{equation}\label{lower}
K_1(1+|x|)^{(\sigma+2)/(m-p)}\leq u(x,t)\leq K_2(1+|x|)^{(\sigma+2)/(m-p)}, \quad {\rm for} \ t>t_D.
\end{equation}
If $p\geq p_c(\sigma)$, the lower bound in \eqref{lower} still holds.
\end{theorem}
This result is substantially stronger than the lower bound obtained in Theorem \ref{th.bound}, since it provides a pointwise lower bound which is uniform on compact sets (the time $t_D>0$ being the only element depending on the compact set $D$ in \eqref{lower}), while the estimate \eqref{bound} only gives a lower bound for the $L^{\infty}$ norm of $u(t)$. Nevertheless, the result of Theorem \ref{th.bound} is an essential ingredient in the proof of \eqref{lower}, together with some technical constructions already employed in the proofs of Theorems \ref{th.outer} and \ref{th.supp}. Observe also that \eqref{lower} implies that, if \eqref{pscond} is satisfied, any asymptotic profile of Eq. \eqref{eq1} must be sandwiched between two stationary functions having the same decay as $|x|\to\infty$.

We plot in Figure \ref{fig2} the result of an experiment illustrating how a solution remains bounded and its maximum only changes very little, while its support expands.

\begin{figure}[ht!]
  \begin{center}
  \includegraphics[width=11cm,height=8cm]{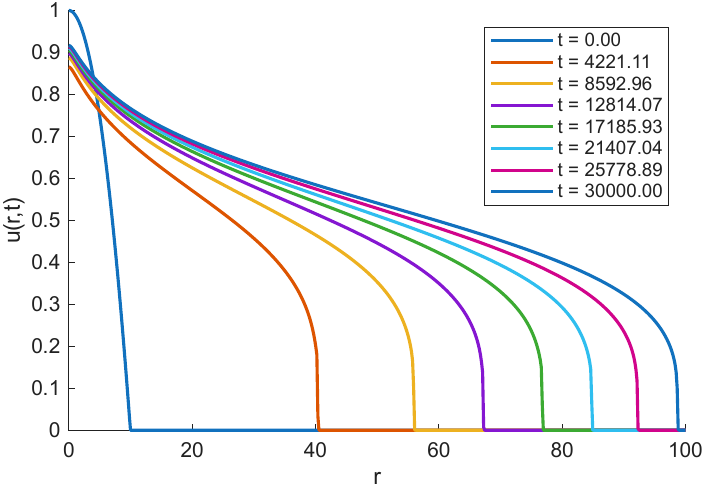}
  \end{center}
  \caption{The evolution of a compactly supported solution at different times. Experiment for $m=5$, $p=2$, $N=3$, $\sigma=-2.5$}\label{fig2}
\end{figure}

The proofs are based mainly on comparison arguments, the suitable sub- and supersolutions employed in the comparison being strongly related to Eq. \eqref{eq0}. This is why, especially in preparation for the proof of Theorem \ref{th.supp}, a rather technical study of some (singular) self-similar solutions to Eq. \eqref{eq0}, employing a phase space analysis associated to a dynamical system, is needed in order to identify such supersolutions.

\medskip

\noindent \textbf{Structure of the paper}. The paper consists of five sections. Theorem \ref{th.bound} is proved in Section \ref{sec.bound}, based on comparison with stationary sub- and supersolutions. The following Section \ref{sec.tech} is the longest one in the paper and contains a number of technical results that are essential for the proof of Theorems \ref{th.outer} and \ref{th.supp}. Relying on an alternative formulation of Eq. \eqref{SSODE} as an autonomous three-dimensional dynamical system together with a careful study of its trajectories, we establish the existence of several self-similar supersolutions and subsolutions to Eq. \eqref{eq1} with suitable properties. Based on these technical results, the proof of Theorems \ref{th.outer} and \ref{th.supp} is performed in Section 4, while the proof of Theorem \ref{th.lower} is given in Section \ref{sec.lower}. We conclude the paper with a discussion of several interesting open questions and critical cases.

\section{Stationary sub- and supersolutions. Proof of Theorem \ref{th.bound}}\label{sec.bound}

In this section we prove Theorem \ref{th.bound}. The proof is based on the construction of stationary sub- and supersolutions to Eq. \eqref{eq1}. We first construct supersolutions, and the construction is strongly related to an explicit unbounded stationary solution to Eq. \eqref{eq0} that is given in the next statement.
\begin{lemma}\label{lem.statabove}
Let $m$, $p$, $N$ and $\sigma$ be as in \eqref{range.exp} and assume that the condition \eqref{pscond} is satisfied. Then the function
\begin{equation}\label{stat.sol}
S(x)=K_0|x|^{(\sigma+2)/(m-p)}, \quad K_0=\left[\frac{(m-p)^2}{m|\sigma+2|[m(N+\sigma)-p(N-2)]}\right]^{1/(m-p)},
\end{equation}
is a stationary solution to Eq. \eqref{eq0}. Moreover, for any $\lambda>1$, the function $\lambda S$ is a supersolution to Eq. \eqref{eq0}.
\end{lemma}
\begin{proof}
Let $S(x)$ be as in \eqref{stat.sol} but with a constant $K_0$ to be determined. Letting $r=|x|$, we find by direct calculation that
\begin{equation*}
\begin{split}
&(S^m)'(r)=K_0^m\frac{m(\sigma+2)}{m-p}r^{m(\sigma+2)/(m-p)-1},\\
&(S^m)''(r)=K_0^m\frac{m(\sigma+2)}{m-p}\left(\frac{m(\sigma+2)}{m-p}-1\right)r^{m(\sigma+2)/(m-p)-2},\\
&r^{\sigma}S^p(r)=K_0^pr^{\sigma+p(\sigma+2)/(m-p)}.
\end{split}
\end{equation*}
Observing that
$$
\sigma+\frac{p(\sigma+2)}{m-p}=\frac{m(\sigma+2)}{m-p}-2=\frac{m\sigma+2p}{m-p},
$$
we find
\begin{equation}\label{interm2}
\begin{split}
-\Delta S^m(x)&-|x|^{\sigma}S^p(x)=-(S^m)''(r)-\frac{N-1}{r}(S^m)'(r)-|x|^{\sigma}S^p(r)\\
&=\left[-K_0^m\frac{m(\sigma+2)}{m-p}\left(\frac{m(\sigma+2)}{m-p}+N-2\right)-K_0^p\right]r^{(m\sigma+2p)/(m-p)}\\
&=-K_0^{p}\left[K_0^{m-p}\frac{m(\sigma+2)(m(\sigma+N)-p(N-2))}{(m-p)^2}+1\right]r^{(m\sigma+2p)/(m-p)}.
\end{split}
\end{equation}
Taking into account the conditions \eqref{pscond} and recalling that $\sigma+2<0$, we observe that there exists a value of $K_0$ for which the expression in the right-hand side of \eqref{interm2} is equal to zero, which is exactly the one given in \eqref{stat.sol}. Moreover, it is obvious that if we replace $K_0$ with $\lambda K_0$ for $\lambda>1$, the term in brackets in the last line of \eqref{interm2} becomes negative and thus the whole expression \eqref{interm2} is positive, proving that $\lambda S$ is a supersolution to Eq. \eqref{eq0} for $\lambda>1$.
\end{proof}
Since $\sigma<0$, it is obvious that any supersolution to Eq. \eqref{eq0} is a supersolution to Eq. \eqref{eq1} as well. However, since $S$ is unbounded as $x\to0$, we have to modify $S$ (and correspondingly $\lambda S$ with $\lambda>1$) in order to obtain a family of \emph{bounded} stationary supersolutions to Eq. \eqref{eq1}.
\begin{lemma}\label{lem.statbd}
Let $m$, $p$, $N$, $\sigma$ be as in \eqref{range.exp} and assume that \eqref{pscond} is satisfied. Let $K_0$ be as in \eqref{stat.sol}. Then, the following family of functions
\begin{equation*}\label{famsuper}
\overline{S}_{\lambda}(x)=\lambda K_0(1+|x|)^{(\sigma+2)/(m-p)}, \quad \lambda>1,
\end{equation*}
is composed of bounded supersolutions to Eq. \eqref{eq1}.
\end{lemma}
\begin{proof}
Letting $r=|x|$, we deduce from Lemma \ref{lem.statabove} that, for any $\lambda>1$, $\overline{S}_{\lambda}$ satisfies
$$
-(\overline{S}_{\lambda}^m)''(r)-\frac{N-1}{1+r}(\overline{S}_{\lambda}^m)'(r)-(1+r)^{\sigma}\overline{S}_{\lambda}^p(r)\geq0.
$$
Recalling that $\sigma+2<0$, we further infer that
\begin{equation*}
\begin{split}
-\Delta\overline{S}_{\lambda}(x)-(1+|x|)^{\sigma}\overline{S}_{\lambda}^p(x)&
=-(\overline{S}_{\lambda}^m)''(r)-\frac{N-1}{r}(\overline{S}_{\lambda}^m)'(r)-(1+r)^{\sigma}\overline{S}_{\lambda}^p(r)\\
&\geq\left(\frac{N-1}{1+r}-\frac{N-1}{r}\right)(\overline{S}_{\lambda}^m)'(r)\\
&=\frac{N-1}{r(1+r)}\frac{-m(\sigma+2)}{m-p}(\lambda K_0)^m(1+r)^{m(\sigma+2)/(m-p)-1}\geq0,
\end{split}
\end{equation*}
proving that $\overline{S}_{\lambda}$ is a supersolution to Eq. \eqref{eq1} for any $\lambda>1$, at least at the points where it is regular. The decreasing profile of $\overline{S}_{\lambda}$ as a function of $r$, ensured by the negativity of the exponent $(\sigma+2)/(m-p)$, shows that it is also a supersolution around the peak $r=0$ where it is not differentiable, completing the proof.
\end{proof}
We proceed now to the construction of stationary subsolutions to Eq. \eqref{eq1}. This is done in a single step, without passing through Eq. \eqref{eq0}, by exploiting the properties of the elliptic counterpart of Eq. \eqref{eq1}.
\begin{lemma}\label{lem.statbelow}
Let $m$, $p$, $N$ and $\sigma$ be as in \eqref{range.exp} and let $R>0$. Then, there exists a unique nontrivial and nonnegative solution $U$ to the homogeneous Dirichlet problem
\begin{equation}\label{Dirp}
\left\{\begin{array}{ll}-\Delta U^m-(1+|x|)^{\sigma}U^p=0, & x\in B(0,R),\\ U(x)=0, & x\in\partial B(0,R), \end{array}\right.
\end{equation}
Moreover, for any $\gamma\in(0,1)$, the function $U_{\gamma}:=\gamma U$ is a subsolution to the same Dirichlet problem \eqref{Dirp}.
\end{lemma}
\begin{proof}
The existence of a unique nontrivial and nonnegative solution $U$ to \eqref{Dirp} follows from the general theory of sublinear elliptic equations given in \cite{BO86}. Indeed, letting
$$
f(x,w)=(1+|x|)^{\sigma}w^{p/m},
$$
the function $w\mapsto f(x,w)$ is obviously continuous on $[0,\infty)$ and $w\mapsto f(x,w)/w$ is decreasing on $(0,\infty)$ since $p<m$. Moreover, the function $x\mapsto f(x,w)$ belongs to $L^{\infty}(B(0,R))$ and $f(x,w)\leq 1+w$ for any $x\in B(0,R)$ and $w\geq0$. Since
$$
a_0(x):=\lim\limits_{w\to0}\frac{f(x,w)}{w}=\infty, \quad a_{\infty}(x):=\lim\limits_{w\to\infty}\frac{f(x,w)}{w}=0,
$$
it follows from \cite{BO86} that the Dirichlet problem
\begin{equation}\label{Dirpbis}
\left\{\begin{array}{ll}-\Delta w-(1+|x|)^{\sigma}w^{p/m}=0, & x\in B(0,R),\\ w(x)=0, & x\in\partial B(0,R), \end{array}\right.
\end{equation}
admits a unique nontrivial and nonnegative solution. Letting $U=w^{1/m}$ in \eqref{Dirpbis}, we deduce that $U$ is the unique nontrivial and nonnegative solution to the Dirichlet problem \eqref{Dirp}, as claimed. Moreover, letting $\gamma\in(0,1)$, we have
\begin{equation*}
\begin{split}
-\Delta U_{\gamma}^m(x)-(1+|x|)^{\sigma}U_{\gamma}^p(x)&=-\gamma^m\Delta U^m(x)-\gamma^p(1+|x|)^{\sigma}U^p(x)\\
&=(\gamma^m-\gamma^p)(1+|x|)^{\sigma}U^p(x)<0,
\end{split}
\end{equation*}
since $m>p$ and $\gamma\in(0,1)$. This shows that $U_{\gamma}=\gamma U$ is indeed a subsolution to the Dirichlet problem \eqref{Dirp} for any $\gamma\in(0,1)$.
\end{proof}
We are now in a position to prove the first boundedness result of this paper.
\begin{proof}[Proof of Theorem \ref{th.bound}]
We divide the proof into two parts, corresponding respectively to the upper and lower bound in \eqref{bound}.

\medskip

\noindent \textbf{Upper bound.} Let $m$, $p$, $N$, $\sigma$ be as in \eqref{range.exp} and $u_0$ as in \eqref{icond} and let $u$ be the solution to the Cauchy problem \eqref{eq1}-\eqref{ic}. As a first step, assume that $\sigma$ and $p$ are such that the conditions \eqref{pscond} are fulfilled. We infer from \eqref{icond} that there is $R_0>0$ such that
\begin{equation}\label{interm3}
u_0(x)\leq\overline{S}_{1}(x)=K_0(1+|x|)^{(\sigma+2)/(m-p)}, \quad {\rm for} \ |x|>R_0.
\end{equation}
Fix $\lambda>1$ sufficiently large such that, for any $x\in B(0,R_0)$, we have
\begin{equation}\label{interm4}
\overline{S}_{\lambda}(x)\geq\overline{S}_{\lambda}(R_0)=\lambda K_0(1+R_0)^{(\sigma+2)/(m-p)}>\|u_0\|_{L^{\infty}(B(0,R_0))},
\end{equation}
where we recall that $\overline{S}_{\lambda}$ is defined in Lemma \ref{lem.statbd} and is a decreasing function of $r=|x|$. Combining \eqref{interm3} and \eqref{interm4}, we deduce that
$$
u_0(x)\leq\overline{S}_{\lambda}(x), \quad x\in\real^N,
$$
and the comparison principle ensures that
\begin{equation}\label{upper}
u(x,t)\leq\overline{S}_{\lambda}(x)=\lambda K_0(1+|x|)^{(\sigma+2)/(m-p)}\leq\lambda K_0, \quad (x,t)\in\real^N\times(0,\infty),
\end{equation}
giving the uniform upper bound.

Now suppose that $p\in(1,m)$ and $\sigma<-2$ are such that \eqref{pscond} is no longer satisfied, that is, $p>p_c(\sigma)$. Recalling the definition of $p_c(\sigma)$ from \eqref{pscond}, we remark that $p_c(\sigma)$ is an increasing function of $\sigma$ and
$$
\lim\limits_{\sigma\to-2}p_c(\sigma)=m.
$$
It follows that there is $\overline{\sigma}\in(\sigma_0,-2)$ such that $\overline{\sigma}>\sigma$ and $p<p_c(\overline{\sigma})$. We thus deduce that the conditions \eqref{pscond} are satisfied for $p$ and $\overline{\sigma}$ and we can apply the first step, finding a suitable supersolution $\overline{S}_{\lambda}$ corresponding to Eq. \eqref{eq1} with $\overline{\sigma}$ in place of $\sigma$. Since
$$
(1+|x|)^{\overline{\sigma}}\geq(1+|x|)^{\sigma}, \quad x\in\real^N,
$$
we find that $\overline{S}_{\lambda}$ is a supersolution to Eq. \eqref{eq1} with exponent $\sigma$ as well. Thus, the upper bound follows again by comparison with the supersolution $\overline{S}_{\lambda}$, completing the proof of the upper bound in the general case.

\medskip

\noindent \textbf{Lower bound.} Assume first that $u_0(x)>0$ for $x\in B(0,2\delta)$ for some $\delta>0$. Let $U$ be the unique nontrivial and nonnegative solution to the homogeneous Dirichlet problem \eqref{Dirp} posed on the smaller ball $B(0,\delta)$. Let $\gamma\in(0,1)$ be sufficiently small such that
$$
\gamma\|U\|_{L^{\infty}(B(0,\delta))}\leq\inf\{u_0(x): x\in B(0,\delta)\}.
$$
This means that the subsolution $U_{\gamma}=\gamma U$ lies below $u_0$ in $B(0,\delta)$. A simple comparison argument shows, as a first step, that the solution $u$ to the Cauchy problem \eqref{eq1}-\eqref{ic} lies above the solution to the evolution Dirichlet problem
\begin{equation*}
\left\{\begin{array}{ll}u_t=\Delta u^m+(1+|x|)^{\sigma}u^p, & (x,t)\in B(0,\delta)\times(0,\infty),\\
u(x,0)=u_0(x), & x\in B(0,\delta), \\ u(x,t)=0, & (x,t)\in\partial B(0,\delta)\times(0,\infty).\end{array}\right.
\end{equation*}
Moreover, the solution to the Dirichlet problem lies (at any $t>0$) above the stationary subsolution $U_{\gamma}$. We thus deduce the lower bound
$$
\|u(t)\|_{\infty}\geq\gamma\|U\|_{\infty}, \quad t\in(0,\infty).
$$
Let us now pick a general $u_0$ as in \eqref{icond}. Since the solution $u$ to the Cauchy problem is a supersolution to the porous medium equation $u_t=\Delta u^m$ (with the same initial condition $u_0$), the standard theory of the porous medium equation (see for example \cite[Proposition 9.19]{VPME}) ensures that there exist $\delta>0$ and a time $t_0>0$ such that $u(x,t_0)>0$ for $x\in B(0,2\delta)$. We can thus start the previous comparison step at $t=t_0$ and establish the desired lower bound, completing the proof.
\end{proof}

\section{Some self-similar solutions and subsolutions}\label{sec.tech}

This is the most technical section of the paper. In order to proceed with the proof of Theorems \ref{th.outer} and \ref{th.supp}, we have to establish the existence of suitable solutions and subsolutions in self-similar form to Eq. \eqref{eq0}, and then employ them in order to compare with solutions to Eq. \eqref{eq1}. We thus look for forward self-similar solutions of the form \eqref{SSS}, with exponents $\alpha$ and $\beta$ given by \eqref{SSexp}. Substituting the ansatz \eqref{SSS} into Eq. \eqref{eq0}, we obtain by standard calculations that the profiles $f$ solve the differential equation
\begin{equation}\label{SSODE}
(f^m)''(\xi)+\frac{N-1}{\xi}(f^m)'(\xi)+\alpha f(\xi)+\beta\xi f'(\xi)+\xi^{\sigma}f^p(\xi)=0, \quad \xi>0,
\end{equation}
where we recall that $\xi=|x|t^{-\beta}$. In order to study the rather complicated equation \eqref{SSODE}, we employ a phase space analysis technique. To this end, we introduce the following functions
\begin{equation}\label{PSchange}
X(\xi)=\frac{m}{\alpha}\xi^{-2}f^{m-1}(\xi), \quad Y(\xi)=\frac{m}{\alpha}\xi^{-1}f^{m-2}(\xi)f'(\xi), \quad Z(\xi)=\frac{1}{\alpha}\xi^{\sigma}f^{p-1}(\xi),
\end{equation}
together with the new independent variable defined implicitly by
\begin{equation}\label{PSvar}
\frac{d\eta}{d\xi}=\frac{\alpha}{m}\xi f^{1-m}(\xi)=\frac{1}{\xi X(\xi)}.
\end{equation}
We then find by direct calculation that
\begin{equation*}
\begin{split}
&f'(\xi)=\frac{\alpha}{m}\xi Y(\xi)f(\xi)^{2-m}, \quad (f^m)'(\xi)=\alpha\xi Y(\xi)f(\xi),\\
&(f^m)''(\xi)=\alpha\left(\xi\frac{d}{d\xi}Y(\xi)f(\xi)+\frac{\alpha}{m}\xi^2f(\xi)^{2-m}Y(\xi)^2+Y(\xi)f(\xi)\right),
\end{split}
\end{equation*}
Substituting the previous expressions into \eqref{SSODE} and employing the chain rule in order to compute derivatives with respect to the variable $\eta$ introduced in \eqref{PSvar}, we derive after straightforward algebraic manipulations the autonomous dynamical system
\begin{equation}\label{PSsyst}
\left\{\begin{array}{ll}\dot{X}=X[(m-1)Y-2X],\\
\dot{Y}=-Y^2-\frac{\beta}{\alpha}Y-X-NXY-XZ,\\
\dot{Z}=Z[(p-1)Y+\sigma X].\end{array}\right.
\end{equation}
Note that we are only interested in the region $X\geq0$, $Z\geq0$ of the phase space and that the limiting planes $X=0$ and $Z=0$ are invariant under the system \eqref{PSsyst}.

\subsection{Critical points of the system \eqref{PSsyst}}

The first step in the classification of the trajectories of system \eqref{PSsyst} is to study its flow in the neighborhood of its critical points. We identify two isolated critical points and a critical line:
$$
Q_0=(0,0,0), \quad Q_1=\left(0,-\frac{\beta}{\alpha},0\right), \quad Q_{\gamma}=\left(0,0,\gamma\right), \quad \gamma>0.
$$
We next analyze the flow of the system in a neighborhood of $Q_0$ and $Q_1$ (the critical points $Q_{\gamma}$ with $\gamma>0$ will not be used in the sequel and are thus omitted from the analysis). The key point is $Q_1$, while $Q_0$ is analyzed only to discard it in the forthcoming proofs, thus we start with the analysis of $Q_1$.
\begin{lemma}\label{lem.Q1}
The critical point $Q_1$ is a saddle point with a two-dimensional stable manifold and a one-dimensional unstable manifold, the latter being fully contained in the invariant $Y$-axis. The profiles corresponding to the trajectories entering $Q_1$ on its stable manifold present an interface behavior, in the sense that there exist $\xi_0\in(0,\infty)$ and $\delta\in(0,\xi_0)$ such that
\begin{equation}\label{beh.int}
f(\xi)>0 \ {\rm for} \ \xi\in(\xi_0-\delta,\xi_0), \quad f(\xi_0)=0, \quad (f^m)'(\xi_0)=0.
\end{equation}
\end{lemma}
\begin{proof}
The linearization of the system \eqref{PSsyst} in a neighborhood of $Q_1$ has the matrix
\begin{equation*}
J(Q_{1}) = \begin{pmatrix}
-\frac{(m-1)\beta}{\alpha} & 0 & 0 \\
-1 + \frac{N\beta}{\alpha} & \frac{\beta}{\alpha} & 0 \\
0 & 0 & -\frac{(p-1)\beta}{\alpha}
\end{pmatrix},
\end{equation*}
recalling that $\alpha$ and $\beta$ are defined in \eqref{SSexp}. The eigenvalues of $J(Q_1)$ are
$$
\lambda_1=-\frac{(m-1)\beta}{\alpha}<0, \quad \lambda_2=\frac{\beta}{\alpha}>0, \quad \lambda_3=-\frac{(p-1)\beta}{\alpha}<0,
$$
with corresponding eigenvectors
$$
v_{1} = \left(1,-\frac{\sigma+2+N(m-p)}{m(m-p)},0\right), \quad v_2=(0,1,0), \quad v_3=(0,0,1).
$$
Thus, there exists a two-dimensional manifold tangent to the plane spanned by the eigenvectors $v_1$ and $v_3$ and a one-dimensional unstable manifold tangent to the eigenvector $v_2$. The uniqueness of the unstable manifold, together with the invariance of the $Y$-axis, imply that the unstable manifold is fully contained in the $Y$-axis. We analyze next the two-dimensional stable manifold. We first infer from the first equation of the system \eqref{PSsyst} that, in a neighborhood of $Q_1$, we have
$$
X(\eta)\sim Ce^{-(m-1)\beta\eta/\alpha}, \quad {\rm as} \ \eta\to\infty, \quad C>0.
$$
By applying the inverse function theorem to the implicit definition of $\eta$ in \eqref{PSvar}, we find that
$$
\frac{\xi'(\eta)}{\xi(\eta)}=X(\eta)\sim Ce^{-(m-1)\beta\eta/\alpha}, \quad {\rm as} \ \eta\to\infty,
$$
which readily gives by integration that $\xi(\eta)\to\xi_0\in(0,\infty)$ as $\eta\to\infty$. Thus, the limit $\eta\to\infty$ on the trajectories contained in the stable manifold of $Q_1$ transforms into $\xi\to\xi_0$ from the left in the original independent variable of the equation \eqref{SSODE}. Since $X(\xi)\to0$ as $\xi\to\xi_0$ from the left on the trajectories entering $Q_1$, we readily obtain by reverting \eqref{PSchange} that $f(\xi_0)=0$, while $f(\xi)>0$ in a left-neighborhood of $\xi_0$. Then we also know that
$$
\lim\limits_{\xi\to\xi_0}Y(\xi)=\lim\limits_{\xi\to\xi_0}\frac{m}{\alpha(m-1)}\xi^{-1}(f^{m-1})'(\xi)=-\frac{\beta}{\alpha},
$$
which, together with L'Hopital's rule, readily gives that
$$
\lim\limits_{\xi\to\xi_0}\frac{f^{m-1}(\xi)}{\xi_0-\xi}=-\lim\limits_{\xi\to\xi_0}(f^{m-1})'(\xi)=\frac{\beta(m-1)}{m}\xi_0,
$$
leading to the interface behavior
\begin{equation}\label{interm5}
f(\xi)\sim\left[\frac{\beta(m-1)}{m}\xi_0\right]^{1/(m-1)}(\xi_0-\xi)^{1/(m-1)}, \quad {\rm as} \ \xi\to\xi_0, \ \xi<\xi_0.
\end{equation}
It is immediate to check that the precise behavior \eqref{interm5} implies the condition $(f^m)'(\xi_0)=0$, completing the proof.
\end{proof}
It is now the turn of $Q_0$, whose analysis will be only sketched.
\begin{lemma}\label{lem.Q0}
The critical point $Q_0$ has, in the region of interest $\{(X,Y,Z)\in\real^3: X\geq0, Y\in\real, Z\geq0\}$, a three dimensional center-stable manifold composed of two-dimensional center manifolds with a stable direction of the flow and a one-dimensional stable manifold.
\end{lemma}
\begin{proof}
The linearization of the system in a neighborhood of $Q_0$ has the matrix
\begin{equation*}
J(Q_{0}) = \begin{pmatrix}
0 & 0 & 0 \\
-1 & -\frac{\beta}{\alpha} & 0 \\
0 & 0 & 0
\end{pmatrix},
\end{equation*}
with a one-dimensional stable manifold (corresponding to the second eigenvalue and the eigenvector $(0,1,0)$) and two-dimensional center manifolds (recalling that a center manifold might not be unique). In order to study the flow of the system \eqref{PSsyst} on the center manifolds, we follow the theory in \cite[Chapter 2]{Carr}. Thus, following \cite[Theorem 3, Section 2.5]{Carr}, we can approximate the center manifolds by setting
$$
W:=\frac{\beta}{\alpha}Y+X, \quad {\rm that \ is,} \quad Y=\frac{\alpha}{\beta}(W-X),
$$
and write the system \eqref{PSsyst} in variables $(X,W,Z)$ as follows:
\begin{equation}\label{PSsystQ0}
\left\{\begin{array}{ll}\dot{X}=-\frac{1}{\beta}X^2+\frac{(m-1)\alpha}{\beta}XW,\\[1mm]
\dot{W}=-\frac{\beta}{\alpha}W-\frac{\beta}{\alpha}XZ+\frac{(N-2)\beta-m\alpha}{\beta}X^2-\frac{N\beta-(m+1)\alpha}{\beta}XW-\frac{\alpha}{\beta}W^2,\\[1mm]
\dot{Z}=-\frac{1}{\beta}XZ+\frac{(p-1)\alpha}{\beta}ZW,\end{array}\right.
\end{equation}
where in the first and third equations we took into account that
$$
\frac{(m-1)\alpha}{\beta}+2=\frac{(p-1)\alpha}{\beta}-\sigma=\frac{1}{\beta}.
$$
Following the general strategy in \cite[Theorem 3, Section 2.5]{Carr}, one can derive a further approximation of the center manifolds up to higher orders. However, the flow on any center manifold is given by the reduced system obtained by keeping the dominating terms in the first and third equations of the system \eqref{PSsystQ0}, which yields
\begin{equation}\label{PSsystQ0r}
\left\{\begin{array}{ll}\dot{X}=-\frac{1}{\beta}X^2+O(|(X,Z)|^3),\\[1mm]
\dot{Z}=-\frac{1}{\beta}XZ+O(|(X,Z)^3|).\end{array}\right.
\end{equation}
It is obvious from the positivity of $\beta$, $X$ and $Z$ that the flow of this system in a neighborhood of its origin $(X,Z)=(0,0)$ goes in the stable direction, completing the proof.
\end{proof}

\noindent \textbf{Remark.} An analogous analysis as in the proof of Lemma \ref{lem.Q1}, based on an integration of the system \eqref{PSsystQ0r}, shows that the profiles corresponding to trajectories entering $Q_0$ have the behavior
$$
\lim\limits_{\xi\to\infty}\xi^{-(\sigma+2)/(m-p)}f(\xi)=K\in(0,\infty),
$$
that is, similar to the tail of the stationary solution defined in \eqref{stat.sol}.

\subsection{Critical points at infinity}

In order to complete the understanding of the global dynamics of the system \eqref{PSsyst}, we have to analyze the critical points at infinity by compactifying the phase space via the Poincar\'e hypersphere transformation (following the methodology given in \cite[Section 3.10]{Pe}). We thus introduce the projective variables $(\overline{X}, \overline{Y}, \overline{Z}, W)$ defined by
\begin{equation*}
X =\frac{\overline{X}}{W}, \quad Y =\frac{\overline{Y}}{W}, \quad Z =\frac{\overline{Z}}{W}.
\end{equation*}
The critical points at infinity lie on the equator of the Poincar\'e hypersphere (where $W=0$), defined by $\overline{X}^2 + \overline{Y}^2 + \overline{Z}^2 = 1$. Moreover, they correspond to the solutions of the following algebraic system:
\begin{equation}\label{gensyst}
\begin{cases}
\overline{X}\overline{Q}(\overline{X},\overline{Y},\overline{Z}) - \overline{Y}\overline{P}(\overline{X},\overline{Y},\overline{Z}) &= 0, \\
\overline{X}\overline{R}(\overline{X},\overline{Y},\overline{Z}) - \overline{Z}\overline{P}(\overline{X},\overline{Y},\overline{Z}) &= 0, \\
\overline{Y}\overline{R}(\overline{X},\overline{Y},\overline{Z}) - \overline{Z}\overline{Q}(\overline{X},\overline{Y},\overline{Z}) &= 0,
\end{cases}
\end{equation}
where $\overline{P}$, $\overline{Q}$ and $\overline{R}$ are the homogeneous quadratic polynomials contained in the original vector field:
\begin{align*}
\overline{P}(\overline{X},\overline{Y},\overline{Z}) &= \overline{X}[(m-1)\overline{Y}-2\overline{X}], \\
\overline{Q}(\overline{X},\overline{Y},\overline{Z}) &= -\overline{Y}^2-N\overline{X}\overline{Y}-\overline{X}\overline{Z}, \\
\overline{R}(\overline{X},\overline{Y},\overline{Z}) &= \overline{Z}[(p-1)\overline{Y}+\sigma\overline{X}].
\end{align*}
By replacing the particular forms of $\overline{P}$, $\overline{Q}$, $\overline{R}$ in the system \eqref{gensyst}, taking into account the equator condition and solving the resulting system, we obtain three critical points with a nonzero $\overline{X}$ coordinate
\begin{equation}\label{critinf}
\begin{split}
&P_0=(1,0,0,0), \quad P_1=\left(\frac{m}{\sqrt{m^2+(N-2)^2}},-\frac{N-2}{\sqrt{m^2+(N-2)^2}},0,0\right), \\
&P_2=\left(\frac{(m-p)^2}{A},\frac{(\sigma+2)(m-p)}{A},-\frac{(\sigma+2)[m(N+\sigma)-p(N-2)]}{A},0\right),
\end{split}
\end{equation}
where
$$
A:=\sqrt{(m-p)^2+(\sigma+2)^2(m-p)^2+(\sigma+2)^2[m(N+\sigma)-p(N-2)]^2},
$$
and other three critical points with $\overline{X}=0$, namely
$$
Q_2=(0,1,0,0), \quad Q_3=(0,-1,0,0), \quad Q_4=(0,0,1,0).
$$
The analysis of the critical points with a nonzero $\overline{X}$ coordinate follows from an application of \cite[Theorem 5(a), Section 3.10]{Pe}, thus they are topologically equivalent to finite critical points of the system obtained from \eqref{PSsyst} by implementing the change of variables
\begin{equation*}
x=\frac{1}{X}, \quad y=\frac{Y}{X}, \quad z=\frac{Z}{X}.
\end{equation*}
We then obtain after direct calculations that the system \eqref{PSsyst} is mapped, in variables $(x,y,z)$, into the system
\begin{equation}\label{PSsyst2}
\left\{\begin{array}{ll}\frac{dx}{d\theta}=x(2-(m-1)y), \\ \frac{dy}{d\theta}=-x-(N-2)y-z-my^2-\frac{\beta}{\alpha}xy, \\ \frac{dz}{d\theta}=z(\sigma+2+(p-m)y),\end{array}\right.
\end{equation}
where the new independent variable is $\theta=\ln\,\xi$. Thus, the critical points $P_0$, $P_1$ and $P_2$ listed in \eqref{critinf} are identified with the critical points
$$
P_0=(0,0,0), \quad P_1=\left(0,-\frac{N-2}{m},0\right), \quad P_2=\left(0,\frac{\sigma+2}{m-p},-\frac{(\sigma+2)(N-2)(p_c(\sigma)-p)}{(m-p)^2}\right),
$$
noticing that $P_2$ exists in our domain of interest with $z\geq0$ only when $p<p_c(\sigma)$, which is equivalent to the conditions \eqref{pscond}. The key critical point is $P_1$, while we also analyze $P_0$ only to discard it in later arguments. We do not study the critical point $P_2$ here, as it will not play any role in the forthcoming analysis.
\begin{lemma}\label{lem.P1}
Let $m$, $p$, $N$, $\sigma$ be as in \eqref{range.exp} and assume that the inequalities \eqref{pscond} are satisfied. Then $P_1$ is an unstable node. The trajectories going out of $P_1$ correspond to profiles with a vertical asymptote at $\xi=0$ of the form
\begin{equation}\label{beh.asympt}
\lim\limits_{\xi\to0}\xi^{(N-2)/m}f(\xi)=C\in(0,\infty).
\end{equation}
If $p>p_c(\sigma)$ (where $p_c(\sigma)$ is defined in \eqref{pscond}), the trajectories of the critical point $P_1$ are fully contained in the invariant planes $x=0$ and $z=0$.
\end{lemma}
\begin{proof}
The linearization of the system \eqref{PSsyst2} in a neighborhood of $P_1$ has the matrix
\begin{equation*}
J(P_{1}) = \begin{pmatrix}
\frac{N(m-1)+2}{m} & 0 & 0 \\
-\frac{m(N+\sigma)-p(N-2)}{m(\sigma+2)} & N-2 & -1 \\
0 & 0 & \frac{m(N+\sigma)-p(N-2)}{m}
\end{pmatrix},
\end{equation*}
with two positive eigenvalues in any case and the third eigenvalue being positive for $p<p_c(\sigma)$ and negative for $p>p_c(\sigma)$. We then observe that, if $p<p_c(\sigma)$, the critical point $P_1$ is an unstable node, while if $p>p_c(\sigma)$ the point becomes a saddle with a two-dimensional unstable manifold fully included in the invariant plane $z=0$ and a one-dimensional stable manifold fully included in the invariant plane $x=0$ (the proof of the inclusion follows from the uniqueness of both the stable and the unstable manifold, see for example \cite[Theorem 3.2.1]{GH}). The proof of the stability is thus complete. 

Regarding the behavior of the profiles, we observe that
$$
\lim\limits_{\theta\to-\infty}y(\theta)=\lim\limits_{\theta\to\infty}\frac{Y(\theta)}{X(\theta)}=-\frac{N-2}{m}.
$$
Taking into account that $\theta=\ln\,\xi$, we find that
\begin{equation}\label{interm6}
\lim\limits_{\xi\to0}\frac{\xi f'(\xi)}{f(\xi)}=\lim\limits_{\xi\to0}\frac{(\ln\,f)'(\xi)}{(\ln\,\xi)'}=-\frac{N-2}{m}.
\end{equation}
Since
$$
z(\xi)=\frac{Z(\xi)}{X(\xi)}=\xi^{\sigma+2}f(\xi)^{p-m}\to0, \quad {\rm as} \ \xi\to0,
$$
the negativity of $\sigma+2$ and of $p-m$ implies that $f(\xi)\to\infty$ as $\xi\to0$. We can thus apply L'Hopital's rule in \eqref{interm6} to deduce the local behavior \eqref{beh.asympt}, as claimed.
\end{proof}
The analysis of the critical point $P_0$ is immediate.
\begin{lemma}\label{lem.P0}
The critical point $P_0$ of the system \eqref{PSsyst2} is a saddle point with a two-dimensional stable manifold fully contained in the invariant plane $x=0$ and a one-dimensional unstable manifold fully contained in the invariant plane $z=0$.
\end{lemma}
\begin{proof}
Since the linearization of the system \eqref{PSsyst2} in a neighborhood of $P_0$ has the matrix
\begin{equation*}
J(P_{0}) = \begin{pmatrix}
2 & 0 & 0 \\
-1 & -(N-2) & -1 \\
0 & 0 & \sigma+2
\end{pmatrix}
\end{equation*}
and that we assume $N\geq3$ and $\sigma<-2$, the conclusion follows from the signs of the three eigenvalues $\lambda_1=2$, $\lambda_2=-(N-2)<0$ and $\lambda_3=\sigma+2<0$, the form of their corresponding eigenvectors
$$
v_1=(N,-1,0), \quad v_2=(0,1,0), \quad v_3=(0,-1,N+\sigma),
$$
and the uniqueness of the stable and unstable manifold \cite[Theorem 3.2.1]{GH}.
\end{proof}
In order to analyze the critical points $Q_2$ and $Q_3$, one should introduce the following variables
\begin{equation}\label{varY}
\overline{x}=\frac{X}{Y}, \quad \overline{y}=\frac{1}{Y}, \quad \overline{z}=\frac{Z}{Y},
\end{equation}
and deduce, according to \cite[Theorem 5(b), Section 3.10]{Pe}, that the flow of the system \eqref{PSsyst} in a neighborhood of these critical points is topologically equivalent to the flow in a neighborhood of the origin in the system transformed in the variables \eqref{varY}, that is
$$
\left\{\begin{array}{ll}\pm\overline{x}'=-\overline{x}\left[m+(N-2)\overline{x}+\overline{x}\overline{y}+\frac{\beta}{\alpha}\overline{y}+\overline{x}\overline{z}\right],\\
\pm\overline{y}'=-\overline{y}\left[1+N\overline{x}+\overline{x}\overline{y}+\frac{\beta}{\alpha}\overline{y}+\overline{x}\overline{z}\right],\\
\pm\overline{z}'=-\overline{z}\left[p+(\sigma+N)\overline{x}+\overline{x}\overline{y}+\frac{\beta}{\alpha}\overline{y}+\overline{x}\overline{z}\right],
\end{array}\right.
$$
where the minus sign in front of the derivatives corresponds to $Q_2$ and the plus sign in front of the derivatives corresponds to $Q_3$. We then have
\begin{lemma}\label{lem.Q23}
The critical point $Q_2$ is an unstable node and the critical point $Q_3$ is a stable node. The trajectories stemming from the unstable node $Q_2$ correspond to profiles such that there is $\xi_0\in(0,\infty)$ and $\delta>0$ with
\begin{equation}\label{beh.Q2}
f(\xi_0)=0, \qquad f(\xi)>0 \ {\rm for} \ \xi\in(\xi_0,\xi_0+\delta), \qquad (f^m)'(\xi_0)>0.
\end{equation}
The trajectories entering the stable node $Q_3$ correspond to profiles having a compact support such that there is $\xi_0\in(0,\infty)$ and $\delta\in(0,\xi_0)$ with
\begin{equation}\label{beh.Q3}
f(\xi_0)=0, \qquad f(\xi)>0 \ {\rm for} \ \xi\in(\xi_0-\delta,\xi_0), \qquad (f^m)'(\xi_0)<0.
\end{equation}
\end{lemma}
The proof is based on the local analysis of the origin in the previous system, which is very similar to the one given in detail in, for example, \cite[Lemma 3.3]{IS25} (to which the reader is referred) and is thus omitted here. Let us note that the profiles corresponding to the local behaviors \eqref{beh.Q2} and \eqref{beh.Q3} do not satisfy the contact condition $(f^m)'(\xi_0)=0$ required for a weak solution of the porous medium equation at the interface point, but they are (if extended by zero outside their support) subsolutions to Eq. \eqref{eq0}.

\subsection{Small self-similar subsolutions to Eq. \eqref{eq0}}

In this section, small self-similar subsolutions in the form \eqref{SSS} are constructed. According to Lemma \ref{lem.Q23}, it is sufficient to identify trajectories connecting the critical points $Q_2$ and $Q_3$. 
\begin{proposition}\label{prop.subs}
There is at least one subsolution $\tilde{u}$ to Eq. \eqref{eq0} in the self-similar form \eqref{SSS} whose profile is compactly supported, that is, ${\rm supp}\,f=[\xi_1,\xi_2]\subset(0,\infty)$.
\end{proposition}
\begin{proof}
The proof is identical to that of the analogous result \cite[Lemma 2.3]{ILS25} and we will only give a sketch here for the sake of completeness. In order to find a trajectory of the system \eqref{PSsyst} connecting the critical points $Q_2$ and $Q_3$, we first perform a further change of variable in order to replace $Z$ by $W=XZ$ in the system \eqref{PSsyst}. We obtain the system
\begin{equation}\label{PSsystbis}
\left\{\begin{array}{ll}\dot{X}=X[(m-1)Y-2X],\\
\dot{Y}=-Y^2-\frac{\beta}{\alpha}Y+X-NXY-W,\\
\dot{W}=W[(m+p-2)Y+(\sigma-2)X],\end{array}\right.
\end{equation}
and it is easy to see that the analysis of the critical points $Q_2$ and $Q_3$ is similar to that in Lemma \ref{lem.Q23}. Let us analyze the invariant plane $X=0$, where the system \eqref{PSsystbis} reduces to
\begin{equation}\label{PSplane}
\left\{\begin{array}{ll}
\dot{Y}=-Y^2-\frac{\beta}{\alpha}Y-W,\\
\dot{W}=(m+p-2)YW,\end{array}\right.
\end{equation}
It readily follows from Lemma \ref{lem.Q1} that the critical point $Q_1=(0,-\beta/\alpha)$ is a saddle point for \eqref{PSplane} and its one-dimensional stable manifold $W_s(Q_1)$ consists of a single trajectory tangent to the eigenvector
$$
e_2=\left(1,\frac{(m+p-1)\beta}{\alpha}\right).
$$
We prove next that the unique trajectory in $W_s(Q_1)$ arrives from the critical point $Q_2$. The isocline $\dot{Y}=0$ of the system \eqref{PSplane} is given by
\begin{equation}\label{isoY}
Y^2+\frac{\beta}{\alpha}Y+W=0, \quad {\rm with \ normal \ vector} \quad \overline{n}=\left(2Y+\frac{\beta}{\alpha},1\right),
\end{equation}
and the direction of the flow of the system \eqref{PSplane} across the isocline \eqref{isoY} is given by the sign of the expression
\begin{equation}\label{interm7}
(m+p-2)\left(-Y^2-\frac{\beta}{\alpha}Y\right)Y\leq0, \quad Y\in\left[-\frac{\beta}{\alpha},0\right].
\end{equation}
Taking into account that the slope of the isocline \eqref{isoY} near $Q_1$ is given by
$$
\frac{dW}{dY}\Big|_{Y=-\beta/\alpha}=\frac{\beta}{\alpha}<\frac{(m+p-1)\beta}{\alpha},
$$
we infer from \eqref{interm7} that the trajectory contained in $W_s(Q_1)$ lies forever in the region $\dot{Y}<0$. The inverse function theorem allows us then to parametrize this trajectory as a curve $W=W(Y)$, with derivative
\begin{equation}\label{asympt}
W'(Y)=-\frac{(m+p-2)W(Y)Y}{Y^2+(\beta/\alpha)Y+W(Y)},
\end{equation}
and it is not difficult to show that this curve cannot have a vertical asymptote $W(Y)\to\infty$ as $Y\to Y_0$ for some $Y_0\in(-\beta/\alpha,0]$. Indeed, in such a case, it follows from \eqref{asympt} that in a left neighborhood of $Y_0$ we have
$$
|W'(Y)|\leq(m+p-2)|Y|
$$
thus the growth would be at most linear, contradicting the existence of the vertical asymptote. We infer that the trajectory in $W_s(Q_1)$ arrives from the half-plane $\{Y>0\}$, in which $W(Y)$ is decreasing with $Y$. An argument by contradiction and a direct application of the Poincar\'e-Bendixson's Theorem \cite[Section 3.7]{Pe} prove that the unique trajectory in $W_s(Q_1)$ arrives from $Q_2$. Since $Q_2$ is an unstable node and the trajectory $W(Y)$ is a separatrix for the system \eqref{PSplane}, any trajectory going out of $Q_2$ into the region $W>W(Y)$ remains forever in the region $W>W(Y)$ for $Y\geq-\beta/\alpha$ and cannot enter $Q_1$. An easy argument by contradiction based on the Poincar\'e-Bendixson's Theorem and the monotonicity of both $W$ and $Y$ along such a trajectory gives that all these trajectories starting from $Q_2$ in the region $W>W(Y)$ must enter the stable node $Q_3$. The stability of $Q_2$ and $Q_3$ entails that there are other orbits of the phase space not contained in the plane $X=0$ connecting them. It follows from Lemma \ref{lem.Q23} that the profiles $f$ corresponding to such trajectories after undoing \eqref{PSchange} are subsolutions to Eq. \eqref{eq0} supported on intervals $[\xi_1,\xi_2]\subset(0,\infty)$, completing the proof.
\end{proof}
We plot in Figure \ref{fig1} a typical phase portrait associated to the system \eqref{PSplane} according to the proof of Proposition \ref{prop.subs}.

\begin{figure}[ht!]
  \begin{center}
  \includegraphics[width=11cm,height=8cm]{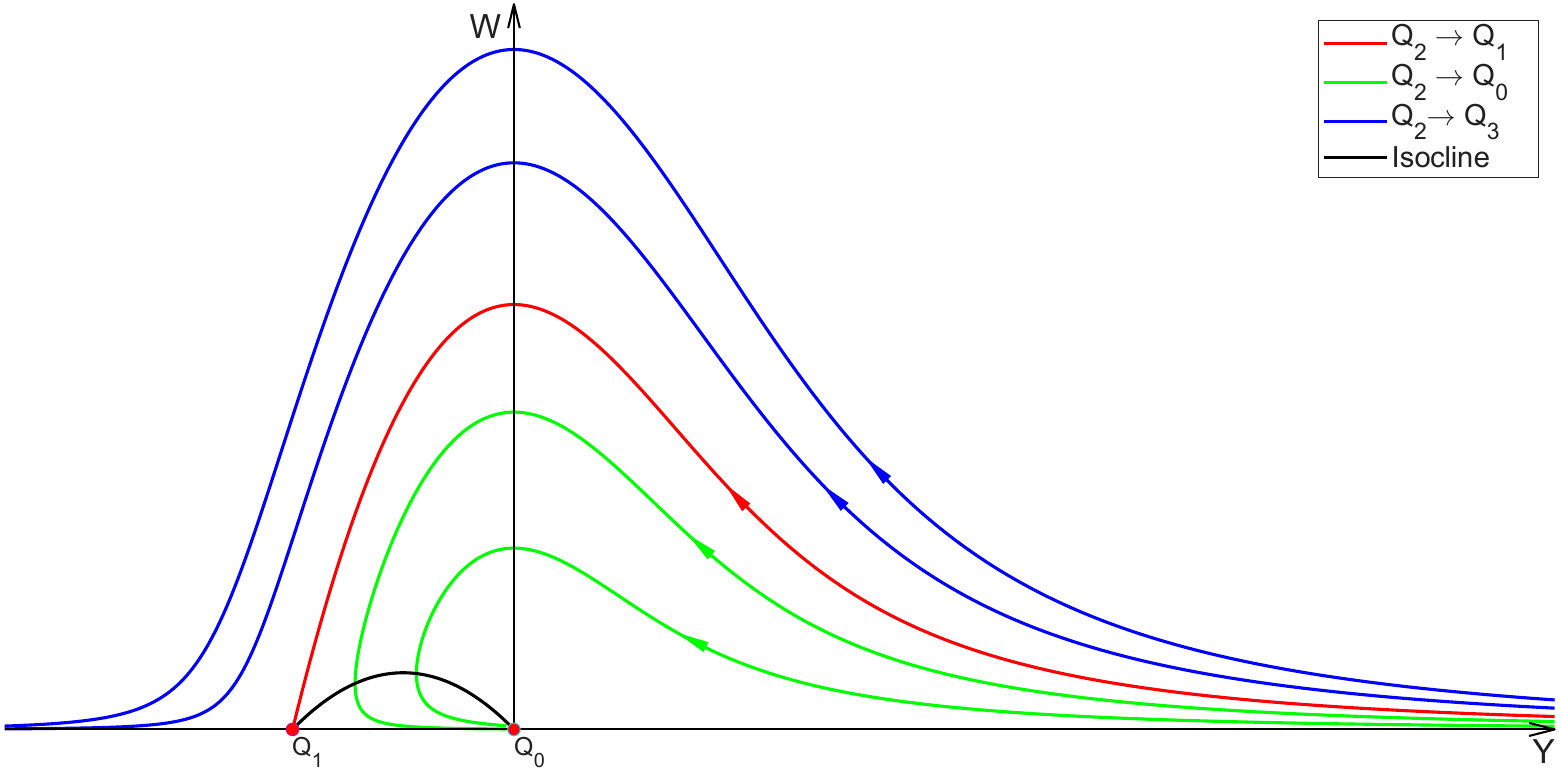}
  \end{center}
  \caption{The phase plane associated to the system \eqref{PSplane}. Experiment for $m=3$, $p=2$, $N=4$, $\sigma=-3$}\label{fig1}
\end{figure}

\subsection{Solutions to Eq. \eqref{eq0} with vertical asymptote}

We next wish to derive large solutions to Eq. \eqref{eq0}, which will be employed later as compactly supported supersolutions to Eq. \eqref{eq1} and will be essential in establishing the finite speed of propagation claimed in Theorem \ref{th.outer}.
\begin{proposition}\label{prop.largesup}
Let $m$, $p$, $N$ and $\sigma$ satisfy \eqref{range.exp} and \eqref{pscond}. Then there are compactly supported self-similar solutions to Eq. \eqref{eq0} in the form \eqref{SSS} whose profiles are decreasing and satisfy the behavior \eqref{beh.asympt} as $\xi\to0$.
\end{proposition}
\begin{proof}
In order to prove the existence of such self-similar solutions, we have to show that there are trajectories of the system \eqref{PSsyst} connecting the critical points $P_1$ and $Q_1$, according to Lemmas \ref{lem.Q1} and \ref{lem.P1}. We first analyze the trajectories of the system contained in the invariant plane $Z=0$, that is, trajectories of the reduced system
\begin{equation}\label{PSsystZ0}
\left\{\begin{array}{ll}\dot{X}=X[(m-1)Y-2X],\\
\dot{Y}=-Y^2-\frac{\beta}{\alpha}Y-X-NXY.\end{array}\right.
\end{equation}
It is an easy consequence of Lemmas \ref{lem.P1} and \ref{lem.Q1} that $P_1$ is an unstable node and $Q_1$ is a saddle point in the system \eqref{PSsystZ0}. The unique trajectory contained in the stable manifold of $Q_1$ enters tangent to the vector
$$
v_1=\left(1,-\frac{\sigma+2+N(m-p)}{m(m-p)}\right).
$$
Since $1<p<p_c(\sigma)$ by \eqref{pscond}, we deduce that
\begin{equation}\label{interm9}
\sigma+2+N(m-p)>\sigma+2+mN\left(1-\frac{N+\sigma}{N-2}\right)=-\frac{(\sigma+2)(mN-N+2)}{N-2}>0.
\end{equation}
Thus, on the one hand, the second component of the eigenvector $v_1$ is negative and, on the other hand, the direction of the flow of the system \eqref{PSsystZ0} across the line $Y=-\beta/\alpha$ is given by the sign of the expression
\begin{equation}\label{interm10}
X\left(\frac{N\beta}{\alpha}-1\right)=X\frac{N(m-p)+\sigma+2}{-(\sigma+2)}>0.
\end{equation}
It follows that the unique trajectory contained in the stable manifold of $Q_1$ reaches $Q_1$ from the half-plane $Y<-\beta/\alpha$ and it lies forever in this half-plane. In particular, $\eta\mapsto X(\eta)$ is a decreasing function. Consider next the isocline $\dot{Y}=0$, that is,
\begin{equation}\label{isoZ}
X=-\frac{Y^2+\frac{\beta}{\alpha}Y}{1+NY},
\end{equation}
with the normal direction
$$
\overline{n}=\left(-1-NY,-2Y-\frac{\beta}{\alpha}-NX\right).
$$
We infer that the flow of the system \eqref{PSsystZ0} across the curve \eqref{isoZ} points in the direction given by the sign of the expression obtained by taking the scalar product of the vector field of the system with $\overline{n}$, which is
$$
E(X,Y)=(-1-NY)X((m-1)Y-2X)<0
$$
in the region $Y<-\beta/\alpha$. Indeed, \eqref{interm10} ensures that $N\beta/\alpha-1>0$, hence $NY+1<0$ for any $Y<-\beta/\alpha$. It is easy to see that the trajectory contained in the stable manifold of $Q_1$ enters the critical point through the region limited between the line $Y=-\beta/\alpha$ and the isocline \eqref{isoZ}, and it remains always in this negatively invariant region (the negative invariance following from the direction of the flow on both the line $Y=-\beta/\alpha$ and the isocline). We obtain that $\eta\mapsto Y(\eta)$ is increasing along the trajectory. The monotonicity of both $X$ and $Y$ and the non-existence of finite critical points in the region $Y<-\beta/\alpha$ imply that the trajectory begins at a critical point at infinity.

Going now to the variables $(x,y,z)$ of the system \eqref{PSsyst2}, we consider the invariant plane $z=0$ (which coincides with $Z=0$), that is,
\begin{equation*}
\left\{\begin{array}{ll}\frac{dx}{d\theta}=x(2-(m-1)y), \\ \frac{dy}{d\theta}=-x-(N-2)y-my^2-\frac{\beta}{\alpha}xy.\end{array}\right.
\end{equation*}
We remark that the line $Y=-\beta/\alpha$ is equivalent to the line $y=-\beta x/\alpha$, while the unique trajectory starting from $P_0$ is tangent to the vector $(1,-N)$, that is, to the line $y=-x/N$. We deduce from \eqref{interm10} that
$$
\frac{N\beta}{\alpha}>1, \quad {\rm that \ is,} \quad -\frac{\beta}{\alpha}<-\frac{1}{N}.
$$
It follows that the trajectory contained in the unstable manifold of $P_0$ goes into the region $y/x=Y>-\beta/\alpha$ and cannot connect to $Q_1$. We thus deduce that the trajectory entering $Q_1$ comes from the critical point $P_1$.

Let us now recall that, under the condition \eqref{pscond}, the point $P_1$ is a full unstable node (of the system \eqref{PSsyst}) according to Lemma \ref{lem.P1}. This instability and a standard continuity argument imply that there are trajectories belonging to the stable manifold of $Q_1$ (in the system \eqref{PSsyst}) coming from $P_1$ and not contained in the plane $Z=0$. Picking such a trajectory and denoting by $f$ the profile corresponding to it after undoing the change \eqref{PSchange}, we get from Lemmas \ref{lem.Q1} and \ref{lem.P1} that $f$ satisfies \eqref{beh.asympt} as $\xi\to0$ and \eqref{beh.int} as $\xi\to\infty$, while the fact that the trajectory remains forever in the half-plane $Y<0$ implies that $f$ is decreasing, completing the proof.
\end{proof}

\noindent \textbf{Remarks. 1. The condition \eqref{pscond} is necessary}. One can observe that, at the level of the plane $Z=0$, the connection $P_1$ to $Q_1$ does not depend strictly on the condition \eqref{pscond}. However, if the condition \eqref{pscond} is not fulfilled, the critical point $P_1$ becomes a saddle for \eqref{PSsyst} according to Lemma \ref{lem.P1} and one might ask whether there still exist profiles with the local behavior \eqref{beh.asympt} as $\xi\to0$. The answer is negative: indeed, if we plug the ansatz \eqref{beh.asympt} in the differential equation \eqref{SSODE}, the three powers of $\xi$ competing as $\xi\to0$ (at least at a formal level) are the following
$$
\xi^{-N}, \quad \xi^{-(N-2)/m}, \quad \xi^{\sigma-p(N-2)/m}.
$$
We note that $N>(N-2)/m$, while, comparing with the third one, we have
$$
\sigma-\frac{p(N-2)}{m}+N=\frac{(N-2)(p_c(\sigma)-p)}{m}<0,
$$
provided $p>p_c(\sigma)$. Thus, the last term is the dominating one in the range $p>p_c(\sigma)$, where \eqref{pscond} is not satisfied. Since there is no possibility of compensating the last term with other terms in \eqref{SSODE}, there is no solution with such a behavior. As we shall see, this non-existence is reflected in a lack of suitable supersolutions in order to prove the exact upper bounds in \eqref{decay} and \eqref{supp} when \eqref{pscond} is not fulfilled.

\medskip 

\noindent \textbf{2. The Fujita exponent plays no role.} A closer inspection of the inequalities \eqref{interm9} and \eqref{interm10} shows that they actually remain true for
$$
1<p<p_F(\sigma):=m+\frac{\sigma+2}{N}.
$$
The exponent $p_F(\sigma)$ plays a fundamental role in the analysis of Eq. \eqref{eq1} when $\sigma>-2$, as shown, for example, in \cite{Qi98, Su02, IS25, IL26}. More precisely, when $\sigma>-2$, $p_F(\sigma)>m$ and this exponent plays the role of a \emph{Fujita exponent} (the name stemming from the seminal paper \cite{Fu66}): it has been established that, for any $p\in(1,p_F(\sigma)]$ if $\sigma\geq0$ and for any $p\in(p_G,p_F(\sigma)]$ if $-2<\sigma<0$ (where $p_G=1-\sigma(m-1)/2<m$), any non-trivial non-negative solution to Eq. \eqref{eq1} blows up in finite time. In sharp contrast with these results, in the range \eqref{range.exp}, we have
$$
p_c(\sigma)-p_F(\sigma)=\frac{(\sigma+2)(mN-N+2)}{N(N-2)}<0,
$$
whence $p_c(\sigma)<p_F(\sigma)<m$ and Theorem \ref{th.bound} shows that $p_F(\sigma)$ no longer acts as a Fujita-type exponent for $\sigma<-2$.

\section{Proof of Theorems \ref{th.outer} and \ref{th.supp}}\label{sec.supp}

In order to complete the proofs of Theorems \ref{th.outer} and \ref{th.supp}, we still need a final technical preparation: manipulating the small subsolutions to Eq. \eqref{eq0} obtained in Proposition \ref{prop.subs} in order to construct small subsolutions (in self-similar form) to Eq. \eqref{eq1}. Indeed, any supersolution to Eq. \eqref{eq0} is directly a supersolution to Eq. \eqref{eq1}, due to the negativity of $\sigma$, which implies that $|x|^{\sigma}>(1+|x|)^{\sigma}$, but this is not true for subsolutions. The next simple result has been proved in \cite[Lemmas 2.1 and 2.2]{ILS25}.
\begin{lemma}\label{lem.subs1}
If $\overline{u}$ is a subsolution to \eqref{eq0} and $c>0$, then
\begin{equation*}
w_{c}(x,t):=\lambda_c^{1/(m-1)}\overline{u}(x,\lambda_ct), \quad \lambda_c=c^{(m-1)/(m-p)},
\end{equation*}
is a subsolution to
\begin{equation}\label{eq0.c}
w_t=\Delta w^m+c|x|^{\sigma}w^p.
\end{equation}
Moreover, if $c\in(0,1)$, then $w_c$ is a subsolution to Eq. \eqref{eq1} in the region
\begin{equation}\label{reg.c}
\mathcal{R}_c:=\left\{(x,t)\in\real^N\times(0,\infty): |x|\geq K(c):=\frac{1}{c^{1/\sigma}-1}\right\}.
\end{equation}
\end{lemma}
\begin{proof}
We provide a formal proof assuming that $\overline{u}$ is a classical subsolution. By direct calculation, we have
\begin{equation*}
\begin{split}
&\partial_tw_c(x,t)=\lambda_c^{m/(m-1)}\partial_t\overline{u}(x,\lambda_ct),\\
&\Delta w_c^m(x,t)=\lambda_c^{m/(m-1)}\Delta \overline{u}^m(x,\lambda_ct),\\
&c|x|^{\sigma}w_c^p(x,t)=\lambda_c^{(m-p)/(m-1)}\lambda_c^{p/(m-1)}|x|^{\sigma}\overline{u}^p(x,\lambda_ct)=\lambda_c^{m/(m-1)}|x|^{\sigma}\overline{u}^p(x,\lambda_ct).
\end{split}
\end{equation*}
These calculations allow us to deduce that
$$
\big(\partial_tw_c-\Delta w_c^m-c|x|^{\sigma}w_c^p\big)(x,t)=\lambda_c^{m/(m-1)}\big(\partial_t\overline{u}-\Delta \overline{u}^m-|x|^{\sigma}\overline{u}^p\big)(x,\lambda_ct)\leq0,
$$
for any $(x,t)\in\real^N\times(0,\infty)$, thus $w_c$ is a subsolution to \eqref{eq0.c}. If $\overline{u}$ is only a weak solution, the proof proceeds in the same way on the weak formulation \eqref{weaksol} by changing variables in the integrals, and we omit here the details. We also obtain that
\begin{equation*}
\partial_tw_c-\Delta w_c^m-(1+|x|)^{\sigma}w_c^p=\big[c|x|^{\sigma}-(1+|x|)^{\sigma}\big]w_c^p.
\end{equation*}
If $(x,t)\in\mathcal{R}_c$, where $\mathcal{R}_c$ is the region defined in \eqref{reg.c}, we have $|x|(c^{1/\sigma}-1)\geq1$, which is equivalent, taking into account the negativity of $\sigma$, to $c|x|^{\sigma}\leq (1+|x|)^{\sigma}$, completing the proof.
\end{proof}
The subsolutions constructed in Lemma \ref{lem.subs1} are sufficient in order to derive estimates from below. Indeed, if we start from a subsolution $\tilde{u}$ in self-similar form given by Proposition \ref{prop.subs} and apply the rescaling introduced in Lemma \ref{lem.subs1} to it, we deduce the existence of a one-parameter family (depending on $c$) of small, compactly supported subsolutions to Eq. \eqref{eq1} maintaining the self-similar form inherited from \eqref{eq0}. Letting $u_0$ be an initial condition satisfying \eqref{icond} and $u$ be the solution to the Cauchy problem \eqref{eq1}-\eqref{ic}, by introducing a sufficient time delay $t_0$ to allow the support of $u(t)$ increase sufficiently, we can place a small subsolution constructed in Lemma \ref{lem.subs1} below $u(t+t_0)$. The rigorous statement of this comparison is the following:
\begin{proposition}\label{prop.subs2}
Let $u_0$ be as in \eqref{icond} and $u$ be the solution to the Cauchy problem \eqref{eq1}-\eqref{ic}. Then there exist $\lambda_*\in(0,1)$, a subsolution
$$
W_{*}(x,t)=\lambda_*^{1/(m-1)}\tilde{u}(x,\lambda_*t), \quad \lambda_*=c_*^{(m-1)/(m-p)},
$$
to Eq. \eqref{eq1} in the self-similar form \eqref{SSS} with a profile $f$ such that ${\rm supp}\,f=[\xi_1,\xi_2]\subset(0,\infty)$ (where $\tilde{u}$ is a subsolution to Eq. \eqref{eq0} given by Proposition \ref{prop.subs}) and a time $t_0>0$ such that $W_*(x,t)\leq u(x,t+t_0-1/\lambda_*)$ for any $t\geq1/\lambda_*$.
\end{proposition}
Let us notice that the time delay $t_0$ is justified by the expansion of the support of $u(t)$ to include the support of $W_*$, while the time delay $1/\lambda_*$ is justified by the requirement that the region $\mathcal{R}_{\lambda^*}$ should include the support of the small subsolution. Actually, the choice of $\lambda_*$ is technically more delicate and depends on the infimum of $u(t_0)$ in $B(0,\xi_2)$, the maximum of $f$ and the endpoints $\xi_1$, $\xi_2$ of the support of $f$. We refrain from giving further details here and refer the reader to the complete proof given in \cite[Proposition 2.5]{ILS25}.

We are now in a position to complete the proofs of both Theorems \ref{th.outer} and \ref{th.supp}. We start with the first of these proofs.
\begin{proof}[Proof of Theorem \ref{th.outer}]
Let $m$, $p$, $N$ and $\sigma$ be as in \eqref{range.exp} and let $u_0$ be an initial condition satisfying \eqref{icond}. Let $W_*$ be the subsolution constructed by Proposition \ref{prop.subs2}. The comparison provided by the same proposition gives, taking into account the self-similar form of $\tilde{u}$, that, for any $t>1/\lambda_*$, we have
$$
u(x,t+t_0-1/\lambda_*)\geq\lambda_*^{1/(m-1)}(\lambda_*t)^{-\alpha}f(|x|(\lambda_*t)^{-\beta})=C_*t^{-\alpha}f(|x|(\lambda_*t)^{-\beta}),
$$
where $C_*$ is a suitable power of $\lambda_*$. The latter is equivalent to
$$
t^{\alpha}u(x,t+t_0-1/\lambda_*)\geq C_*f(\lambda_*^{-\beta}|x|t^{-\beta}),
$$
Letting $s=t+t_0-1/\lambda_*$, which implies $t=s-t_0+1/\lambda_*$, we further get, for any $s>t_0$,
\begin{equation*}
\begin{split}
s^{\alpha}u(x,s)&=\left(\frac{s}{s-t_0+1/\lambda_*}\right)^{\alpha}t^{\alpha}u(x,s)\\
&\geq\left(\frac{s}{s-t_0+1/\lambda_*}\right)^{\alpha}C_*f\left(\lambda_*^{-\beta}|x|s^{-\beta}\left(\frac{s}{s-t_0+1/\lambda_*}\right)^{\beta}\right).
\end{split}
\end{equation*}
Letting $x\in\mathcal{O}_{\delta}(s)$ and taking the supremum over the set $\mathcal{O}_{\delta}(s)$, we readily obtain the lower bound in \eqref{decay}. Indeed, in the right-hand side we obtain the supremum of
$$
\left(\frac{s}{s-t_0+1/\lambda_*}\right)^{\alpha}C_*f\left(\lambda_*\xi\left(\frac{s}{s-t_0+1/\lambda_*}\right)^{\beta}\right), \quad \xi>\delta,
$$
which is not equal to zero and is independent of $s$, since the term still depending on $s$ is bounded for $s$ sufficiently large. Note also that the previous proof is completely independent of the condition \eqref{pscond}.

Assume next that \eqref{pscond} is in force. The proof of Theorem \ref{th.bound} and the negativity of $(\sigma+2)/(m-p)$ entail that there exists $\lambda^*>1$ such that
\begin{equation}\label{interm11}
u(x,t)\leq\lambda^*K_0(1+|x|)^{(\sigma+2)/(m-p)}<\lambda^*K_0|x|^{(\sigma+2)/(m-p)}, \quad (x,t)\in\real^N\times(0,\infty),
\end{equation}
where $K_0$ is defined in \eqref{stat.sol}. Since $\alpha/\beta=-(\sigma+2)/(m-p)$, we can further write
\begin{equation}\label{interm12}
\lambda^*K_0|x|^{(\sigma+2)/(m-p)}=\lambda^*K_0t^{-\alpha}|(xt^{-\beta})|^{(\sigma+2)/(m-p)}.
\end{equation}
Letting now $t>0$ and $x\in\mathcal{O}_{\delta}(t)$, we deduce from the negativity of $(\sigma+2)/(m-p)$ that
$$
t^{\alpha}u(x,t)\leq\lambda^*K_0\delta^{(\sigma+2)/(m-p)}, \quad t\in(0,\infty), \quad x\in\mathcal{O}_{\delta}(t),
$$
completing the proof.
\end{proof}

\noindent \textbf{Remark. Necessity of the condition \eqref{pscond}.} Note that the final part of the proof cannot be performed without the condition \eqref{pscond}, since in the last estimates we need the exponents $\alpha$ and $\beta$ associated exactly to the given $\sigma<-2$, in order to get the upper bound. Thus, we can still write \eqref{interm11} in the general case (removing the condition \eqref{pscond}) by replacing $\sigma$ with a higher exponent $\sigma'\in(\sigma,-2)$ for which \eqref{pscond} holds true, in the same way as we did in the proof of Theorem \ref{th.bound}, but then \eqref{interm12} is only valid with the exponents $\alpha'$ and $\beta'$ corresponding to the exponent $\sigma'$ and not to the original exponent $\sigma$, leading to a weaker upper bound.

We continue with the proof of Theorem \ref{th.supp}.
\begin{proof}[Proof of Theorem \ref{th.supp}]
Let $m$, $p$, $N$ and $\sigma$ be as in \eqref{range.exp} and $u_0$ be a compactly supported function satisfying \eqref{icond}. The comparison with the small compactly supported subsolution $W_*$ established in Proposition \ref{prop.subs2} readily leads to the lower bound in Theorem \ref{th.supp}. We leave the details to the reader, noting that, once more, this estimate is valid independently of the condition \eqref{pscond}.

Assume next that \eqref{pscond} holds true. We infer from Proposition \ref{prop.largesup} that there is at least one compactly supported self-similar solution to Eq. \eqref{eq0} (and thus supersolution to Eq. \eqref{eq1}) such that
\begin{equation}\label{super}
U(x,t)=t^{-\alpha}f(|x|t^{-\beta}), \quad \lim\limits_{\xi\to0}\xi^{(N-2)/m}f(\xi)=K\in(0,\infty).
\end{equation}
Let $\xi_0\in(0,\infty)$ be such that ${\rm supp}\,f=[0,\xi_0]$. We then infer from \eqref{super} that ${\rm supp}\,U(t)=B(0,\xi_0t^{\beta})$ for any $t>0$. Moreover, given $x\in\real^N$, we deduce from \eqref{beh.asympt} that
$$
\lim\limits_{t\to\infty}U(x,t)=\lim\limits_{t\to\infty}t^{-\alpha}(|x|t^{-\beta})^{-(N-2)/m}=|x|^{-(N-2)/m}\lim\limits_{t\to\infty}t^{\beta(N-2)/m-\alpha}=\infty,
$$
since
$$
\frac{\beta(N-2)}{m}-\alpha=\frac{(m-p)(N-2)+m(\sigma+2)}{-\sigma(m-1)-2(p-1)}=\frac{m(N+\sigma)-p(N-2)}{-\sigma(m-1)-2(p-1)}>0,
$$
as follows from the condition \eqref{pscond}. Letting
$$
R_0:=\sup\{|x|\in\real^N: u_0(x)>0\},
$$
we can choose $\tau_0>0$ sufficiently large such that $\xi_0\tau_0^{\beta}>2R_0$ and $U(R_0,\tau_0)>\|u_0\|_{\infty}$. Since $U$ has a decreasing profile according to Proposition \ref{prop.largesup}, we deduce that $U(x,\tau_0)>u_0(x)$, for any $x\in B(0,R_0)$ and the comparison principle ensures that
$$
u(x,t)\leq U(x,t+\tau_0), \quad t\in(0,\infty).
$$
In particular, we find that
$$
R(t)\leq\xi_0(t+\tau_0)^{\beta}<c_2t^{\beta}, \quad t\in(1,\infty),
$$
for some $c_2>0$ large enough, establishing the claimed upper bound.

Finally, if $\sigma$ is sufficiently small such that the condition \eqref{pscond} is no longer valid, there exists $\sigma'\in(\sigma,-2)$ such that $1<p<p_c(\sigma')$. Thus, there exists a self-similar large supersolution of the form \eqref{super} for Eq. \eqref{eq1} with exponent $\sigma'$ and thus also for Eq. \eqref{eq1} with the original exponent $\sigma$. The same comparison as above ensures the finite speed of propagation of the support of $u(t)$, but the upper bound is derived in terms of the exponent $\beta'$ corresponding to $\sigma'$ and not in terms of the self-similarity exponent $\beta$ associated with the original $\sigma$.
\end{proof}

\section{Proof of Theorem \ref{th.lower}}\label{sec.lower}

We are now ready to prove Theorem \ref{th.lower}. The proof employs both the outcome of Theorem \ref{th.bound} and the technical constructions performed in Section \ref{sec.tech} and already used in the previous section. We begin with a preparatory lemma.
\begin{lemma}\label{lem.rad}
Let $u_0$ be a bounded radially symmetric function with non-increasing profile in radial variables and $u$ be the solution to the Cauchy problem \eqref{eq1}-\eqref{ic}. Then $u(t)$ is radially symmetric with non-increasing profile for any $t\in(0,\infty)$.
\end{lemma}
\begin{proof}
The radial symmetry of $u(t)$ for $t>0$ is obvious from the rotational invariance of Eq.~\eqref{eq1}. We provide a formal proof of the monotonicity, which can be made rigorous by a standard approximation argument. We derive the equation solved by $\partial_r u$, $r=|x|$. To this end, we start from Eq.~\eqref{eq1} written in radially symmetric variables as
\begin{equation}\label{eq1.rad}
	\partial_t u(t,r) = \partial_r^2 u^m(t,r) + \frac{N-1}{r} \partial_r u^m(t,r)+(1+r)^{\sigma} u^p(t,r), 
\end{equation}
and differentiate~\eqref{eq1.rad} with respect to $r$ to find, by direct calculation, that $w:=\partial_r u$ solves
\begin{equation}\label{eq1.deriv}
	\begin{split}
		\partial_t w &=mu^{m-1} \partial_r^2 w + 3m(m-1)u^{m-2}w\partial_r w + m(m-1)(m-2)u^{m-3}w^3\\
		& \quad -\frac{m(N-1)}{r^2}u^{m-1}w + \frac{m(N-1)(m-1)}{r}u^{m-2}w^2\\
		& \quad + \frac{m(N-1)}{r}u^{m-1} \partial_r w +p(1+r)^{\sigma}u^{p-1}w +\sigma(1+r)^{\sigma-1}u^p.
	\end{split}
\end{equation}
This is a degenerate parabolic equation that satisfies the comparison principle on the positivity set of $u$. The negativity of $\sigma$ ensures that $w\equiv0$ is a strict supersolution to~\eqref{eq1.deriv} and we infer by comparing on the positivity set of $u$ that $w(t)\leq0$ for any $t\geq0$, provided $w_0=\partial_r u_0\leq0$. This completes the proof.
\end{proof}
With this preparation, we can proceed to the proof of Theorem \ref{th.lower}.
\begin{proof}[Proof of Theorem \ref{th.lower}]
We split the proof into several steps for the simplicity of the presentation. 

\medskip

\noindent \textbf{Step 1. Lower bound near the origin.} Assume for the moment that $u_0$ is radially symmetric with a non-increasing profile. It follows from Lemma \ref{lem.rad} that $u(t)$ enjoys the same properties for any $t>0$ and, in particular, $\|u(t)\|_{\infty}=u(0,t)$ for any $t\geq0$. Recalling Theorem \ref{th.bound} and its proof, we have shown that there exist $r_0>0$ sufficiently small such that
\begin{equation}\label{infstat}
\inf\{u(x,t): (x,t)\in B(0,r_0)\times(t_0,\infty)\}=\kappa>0.
\end{equation}
We recall that this lower bound follows by comparison with the stationary solution of a homogeneous Dirichlet problem in a double ball $B(0,2r_0)$, thus we can ensure that the infimum in \eqref{infstat} is strictly positive.

\medskip 

\noindent \textbf{Step 2. A stationary subsolution outside $B(0,r_0)$.} In this step, we consider the function 
\begin{equation}\label{subs}
V(x):=K(1+|x|)^{(\sigma+2)/(m-p)}, 
\end{equation}
with $K>0$ to be chosen later. We request $V$ to be a subsolution to Eq. \eqref{eq1} in the exterior domain $\real^N\setminus B(0,r_0)$. We calculate (with $r=|x|$)
\begin{equation*}
\begin{split}
\Delta V^m(x)&=(V^m)_{rr}(x)+\frac{N-1}{r}(V^m)_r(x)=K^m\frac{m(\sigma+2)}{m-p}(1+r)^{m(\sigma+2)/(m-p)-1}\frac{N-1}{r}\\
&+K^m\frac{m(\sigma+2)}{m-p}\left[\frac{m(\sigma+2)}{m-p}-1\right](1+r)^{(m\sigma+2p)/(m-p)}\\
&=-\frac{K^mm(\sigma+2)}{m-p}(1+r)^{(m\sigma+2p)/(m-p)}\left[-\frac{m(\sigma+2)}{m-p}+1-\frac{(N-1)(1+r)}{r}\right].
\end{split}
\end{equation*}
In view of this calculation and of the fact that 
$$
\sigma+\frac{p(\sigma+2)}{m-p}=\frac{m\sigma+2p}{m-p},
$$
we obtain after easy algebraic manipulations that 
\begin{equation}\label{interm13}
\begin{split}
-\Delta V^m(x)&-(1+|x|)^{\sigma}V^p(x)=K^m(1+r)^{(m\sigma+2p)/(m-p)}\\
&\times\left[\frac{m(\sigma+2)}{m-p}\left(-\frac{m(\sigma+2)}{m-p}+1-\frac{(1+r)(N-1)}{r}\right)-K^{p-m}\right].
\end{split}
\end{equation}
We thus want to choose $K$ such that, on the one hand, $K<\kappa$ (where $\kappa$ is defined in \eqref{infstat}) and, on the other hand, the term in brackets in \eqref{interm13} is negative for $r\geq r_0$. Since the function 
$$
r\longrightarrow\frac{(N-1)(1+r)}{r}
$$
is decreasing with $r$, and observing that $p<p_c(\sigma)$ implies
$$
-\frac{m(\sigma+2)}{m-p}+1-\frac{(N-1)(1+r_0)}{r_0}=\frac{(N-2)(p-p_c(\sigma))}{m-p}-\frac{N-1}{r_0}<0,
$$
it is enough to choose $K>0$ such that 
\begin{equation}\label{Kcond1}
0<K<\min\left\{\kappa,\left[\frac{m(\sigma+2)}{m-p}\left(-\frac{m(\sigma+2)}{m-p}+1-\frac{(1+r_0)(N-1)}{r_0}\right)\right]^{1/(p-m)}\right\}.
\end{equation}
With this choice of $K$, which only depends on $m$, $p$, $N$, $\sigma$ and $r_0$, $V(x)$ becomes a stationary subsolution to Eq. \eqref{eq1} in the complement of the ball $B(0,r_0)$.

\medskip 

\noindent \textbf{Step 3. Comparison between $V$ and $W_{*}$.} Recall next the subsolution 
$$
W_{*}(x,t)=\lambda_*^{1/(m-1)}t^{-\alpha}f(|x|\lambda_*^{-\beta}t^{-\beta})
$$
constructed in Proposition \ref{prop.subs2}, with a profile $f$ supported on an interval $[\xi_1,\xi_2]\subset(0,\infty)$. Let $\xi_*$ be the unique point of maximum of $f$. Then $W_{*}(\cdot,t)$ achieves its maximum value at 
$$
r(t):=\xi_*\lambda_*^{\beta}t^{\beta}, \quad {\rm with} \quad W_{*}(r(t),t)=\lambda_*^{1/(m-1)}t^{-\alpha}f(\xi_*).
$$
Let us impose the following smallness condition on $K$:
\begin{equation}\label{Kcond2}
0<K<f(\xi_*)\xi_*^{-(\sigma+2)/(m-p)}\lambda_*^{1/(m-1)+\alpha}.
\end{equation}
We prove that \eqref{Kcond2} implies 
\begin{equation}\label{interm14}
M_*(t):=W_{*}(r(t),t)=\max\{W_{*}(x,t): x\in\real^N\}>V(r(t)), \quad t>0,
\end{equation}
where $V$ is the explicit stationary subsolution (outside the ball $B(0,r_0)$) constructed in Step 2. Indeed, the condition \eqref{Kcond2} ensures that 
\begin{equation*}
\begin{split}
f(\xi_*)&>\lambda_*^{-1/(m-1)-\alpha}\xi_*^{(\sigma+2)/(m-p)}K=\lambda_*^{-1/(m-1)}K(\xi_*\lambda_*^{\beta})^{-\alpha/\beta}\\
&>\lambda_*^{-1/(m-1)}K\left(\frac{1}{t^{\beta}}+\xi_*\lambda_*^{\beta}\right)^{-\alpha/\beta}\\
&=\lambda_*^{-1/(m-1)}Kt^{\alpha}(1+\xi_*\lambda_*^{\beta}t^{\beta})^{(\sigma+2)/(m-p)}=\lambda_*^{-1/(m-1)}t^{\alpha}V(r(t)),
\end{split}
\end{equation*}
which leads to \eqref{interm14}. Taking into account the radial symmetry of both $W_{*}$ and $V$ together with the compactness in $(0,\infty)$ of the support of $W_*$, we infer from \eqref{interm14} that at any $t>0$ there are at least two intersection points between these two functions. We denote by $r_1(t)$ the unique intersection point such that $r_1(t)\in(0,r(t))$, whose uniqueness follows from the opposite monotonicity of $W_*(\cdot,t)$ and $V$ for $r<r(t)$. We then denote by $r_2(t)\in(0,\infty)$ the first intersection point between $W_*(\cdot,t)$ and $V$ such that $r_2(t)>r(t)$, see Figure \ref{fig3}.

\medskip 

\noindent \textbf{Step 4. Conclusion for radially symmetric solutions.} Choose $K>0$ satisfying both the smallness conditions \eqref{Kcond1} and \eqref{Kcond2}. Recalling the radial symmetry of both $u(\cdot,t)$ and $V$ as well as the choice of $r_0$ and $K$ in \eqref{Kcond1}, we define
$$
s(t):=\sup\{r\in(0,\infty): u(r,t)>V(r)\}>r_0.
$$
If there exists $t_1>0$ such that $s(t_1)=\infty$, then the lower bound in \eqref{lower} is already fulfilled at $t=t_1$ and by comparison on $(r_0,\infty)$ it also remains true for $t\in(t_1,\infty)$. We may thus assume that $s(t)<\infty$ for any $t>0$, and then $s(t)$ is the first intersection point between the radial profiles of $u(\cdot,t)$ and $V$. For example, this is the case when $u_0$ is compactly supported, according to Theorem \ref{th.supp}, see Figure \ref{fig3}. It only remains to prove that $s(t)\to\infty$ as $t\to\infty$. To this end, we pick $K>0$ eventually smaller (it is sufficient, due to the radial monotonicity of both $u(t)$ and $V$, to pick $K$ smaller than the maximum $M_*(1/\lambda_*)$ of $W_{*}(\cdot,1/\lambda_*)$) such that we have $s(t_0)>r_2(1/\lambda_*)>r(1/\lambda_*)$, where $t_0$ and $\lambda_*$ are the ones from Proposition \ref{prop.subs2}. The choice of $K$ according to \eqref{Kcond1} and the fact that $V(x)\leq K$ for any $x\in\real^N$ ensure, on the one hand, that
$$
u(x,t)\geq\kappa>K>V(x), \quad |x|=r_0, \quad t>0,
$$
and, on the other hand, that $V$ is a subsolution to Eq. \eqref{eq1} for $|x|\geq r_0$ and $t>0$. A comparison argument in the region $|x|\geq r_0$, $t>1/\lambda_*$, together with Proposition \ref{prop.subs2}, imply that
$$
s\left(t+t_0-\frac{1}{\lambda_*}\right)>r_2(t), \quad t>\frac{1}{\lambda_*}.
$$
Since, from Step 3, 
$$
r_2(t)>r(t)=\xi_*\lambda_*^{\beta}t^{\beta}\to\infty, \quad {\rm as} \ t\to\infty,
$$
we deduce that $s(t)\to\infty$ as $t\to\infty$. Thus, picking any compact set $D\subset\real^N$, there exists $t_D>0$ such that $D\subset B(0,s(t))$ for any $t\geq t_D$. The lower bound \eqref{lower} follows then from the choice of $s(t)$.

In order to ease the reading of the previous technical argument, we draw in Figure \ref{fig3} the relative positions of the radially symmetric functions involved in the previous steps and their intersection points.

\begin{figure}[htbp]
    \centering
\begin{tikzpicture}[scale=1.2, thick, >=latex]

    \definecolor{curvaSup}{RGB}{214, 39, 40}
    \definecolor{curvaInf}{RGB}{31, 119, 180}
    \definecolor{curvaCampana}{RGB}{44, 160, 44}
    \definecolor{lineasApoyo}{RGB}{150, 150, 150}

    \draw[ thick, ->] (0.5, 0) -- (11.5, 0) ;
    \draw[ thick, ->] (1, -0.5) -- (1, 5.0);

    \draw[very thick, name path=curvaF, color=curvaInf, domain=1:11.5, samples=150]
        plot (\x, {2.5*exp(-0.25*\x)}) node[above right] {$V$};

    \draw[very thick, name path=curvaG, color=curvaSup] (1, 4.5) .. controls (4, 4.4) and (7, 3.0) .. (10, 0)
        node[above right, pos=0.6] {$u$};

    \draw[very thick, name path=curvaU, color=curvaCampana, domain=3.9:6.98, samples=100]
        plot (\x, {-(\x - 5.45)^2 + 2.0525});
    \node[color=curvaCampana, above] at (5.45, 2.0525) {$W_*$};

    \path[name intersections={of=curvaF and curvaU, by={H1, H2}}];
    \path[name intersections={of=curvaF and curvaG, by={S}}];

    \draw[dashed, color=lineasApoyo] (2, 0) -- (2, 4.45);
    \node[below] at (2, -0.1) {$r_0$};

    \draw[dashed, color=lineasApoyo] (H1) -- (H1 |- 0,0);
    \filldraw (H1) circle (1.2pt);
    \node[below] at (H1 |- 0,-0.1) {$r_1(t)$};

    \draw[dashed, color=lineasApoyo] (H2) -- (H2 |- 0,0);
    \filldraw (H2) circle (1.2pt);
    \node[below] at (H2 |- 0,-0.1) {$r_2(t)$};

\draw[dashed, color=lineasApoyo] (S) -- (S |- 0,0);
    \filldraw (S) circle (1.2pt);
    \node[below] at (S |- 0,-0.1) {$s(t)$};

\end{tikzpicture}
    \caption{Geometric configuration of the functions and intersection points employed in the proof of Theorem \ref{th.lower}}
    \label{fig3}
\end{figure}
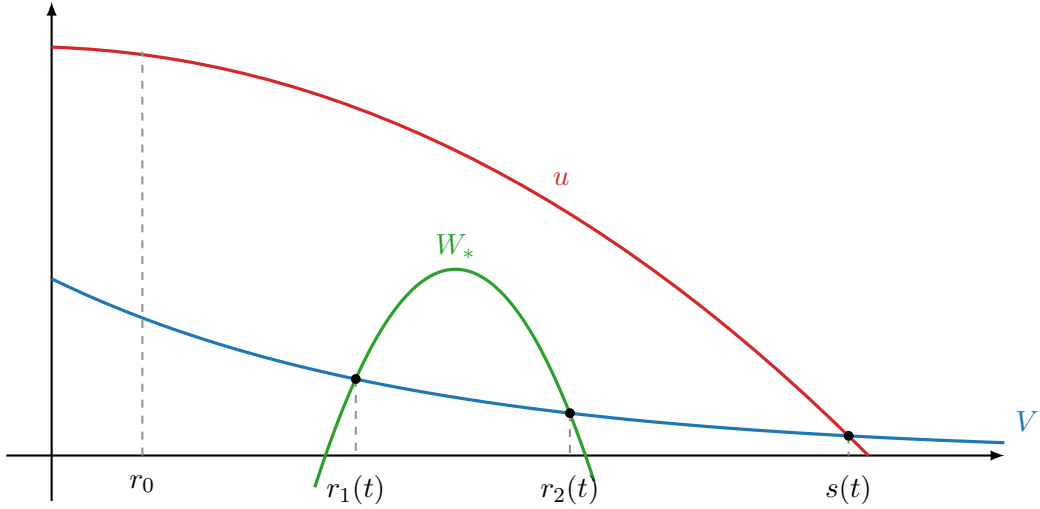

\medskip

\noindent \textbf{Step 5. End of the proof.} Removing now the condition of radial symmetry from $u_0$, it is well-known from basic results on the porous medium equation that there exists $\tau>0$ such that $u(\tau)>0$ on a ball $B(0,2\delta)$ for some $\delta>0$. We can then place below $u(\tau)$ a small initial condition $v_0$ which is radially symmetric, non-increasing with respect to $r$ and supported on $B(0,\delta)$. The comparison principle gives $u(x,t+\tau)\geq v(x,t)$ for any $(x,t)\in\real^N\times(0,\infty)$, while the lower bound in \eqref{lower} applies to $v$, according to the previous steps. Then, it also applies to $u$ with a simple time shift, as claimed.

Observe that all the previous steps of the proof can be performed for any $p\in(1,m)$, without imposing the conditions \eqref{pscond}. Indeed, this condition has been employed only once, in the choice of $K$ in \eqref{Kcond1}. However, if the big term in brackets in \eqref{interm13} is directly negative, then we simply choose $K<\kappa$ and the proof goes on identically. 

Finally, when the conditions \eqref{pscond} are satisfied, the upper bound in \eqref{lower} is an immediate consequence of \eqref{upper}, completing the proof.
\end{proof}

\noindent \textbf{Remark.} Note that, if $p\geq p_c(\sigma)$, we cannot obtain the upper bound in \eqref{lower}. Instead, we can only obtain a weaker upper bound with a supersolution $\overline{S}_{\lambda}$ corresponding to Eq. \eqref{eq1} with a suitable $\overline{\sigma}>\sigma$ in place of $\sigma$, as already shown in the proof of Theorem \ref{th.bound}.

\section*{Final discussion: extensions and open questions}

The results of this paper open a number of interesting questions regarding the properties of solutions to Eq. \eqref{eq1} which deserve further investigation. We briefly discuss some of them, pointing out the extra difficulties and issues arising in their analysis.

\medskip 

\noindent \textbf{1. Critical low dimensions.} As commented in the Introduction, dimensions $N=1$ and $N=2$ are special and have been intentionally excluded from this work. Indeed, if we replace the weight $(1+|x|)^{\sigma}$ with a compactly supported weight 
$$
a(x)=\begin{cases}
  1, & \mbox{if } x\in B(0,L), \\
  0, & \mbox{if } |x|\geq L,
\end{cases}
$$
it was shown in \cite{FdP18} (dimension $N=2$) and \cite{FdP20} (dimension $N=1$) that solutions to Eq. \eqref{eq2} exhibit grow-up as $t\to\infty$ (that is, they are unbounded but do not blow up in finite time) for any $p\in(1,m)$ in dimension $N=2$ and for any $p\in(1,(m+1)/2)$ in dimension $N=1$. Moreover, in dimension $N=1$ and for $(m+1)/2\leq p\leq m+1$, finite time blow-up occurs for any non-trivial solution to Eq. \eqref{eq2} (see \cite{FdPV06}). A simple comparison suggests that the same phenomena should occur for solutions to Eq. \eqref{eq1}, in sharp contrast with the uniform boundedness established in Theorem \ref{th.bound}. However, this simple comparison is not sufficient to establish the grow-up rates (or the blow-up rates in dimension $N=1$ for $p>(m+1)/2$).

\medskip

\noindent \textbf{2. Precise large time behavior.} The problem of proving convergence to a precise asymptotic profile seems to be extremely difficult, especially when the conditions \eqref{pscond} are not satisfied, when even a plausible candidate for the asymptotic profile is unknown. At an intuitive level, due to the different time scales in inner sets and outer sets emphasized in Theorems \ref{th.bound} and \ref{th.outer}, one may conjecture a behavior involving three regions: one where the large time behavior is given by a stationary solution (on compact sets), one where the asymptotic profile could be a (local) self-similar solution with compact support (in outer sets, near the limit of the positivity region of $u(t)$) and a large transition region. Such a complex asymptotic description has been obtained in the case of reaction-diffusion equations by Herraiz in \cite{He99} (see also \cite{HV92} for an absorption problem), while a simpler description, considering only the inner and outer sets (without the intermediate region) is given for a nonhomogeneous porous medium equation in \cite{KRV10}. We refrain from pursuing this seemingly complex problem and leave it as open for the future.

\medskip 

\noindent \textbf{3. The case $\sigma=-2$}. Since it separates the range $\sigma>-2$ (where self-similar grow-up of solutions has been established in \cite{ILS25} with an algebraic grow-up rate) from the range $\sigma<-2$ (where uniform boundedness of solutions is established in the present paper), a natural conjecture in the critical case $\sigma=-2$ is that the solutions to Eq. \eqref{eq1} exhibit logarithmic grow-up as $t\to\infty$. Moreover, two of the authors proved in \cite{IS23} the existence of a unique unbounded and compactly supported self-similar solution to Eq. \eqref{eq0} with $\sigma=-2$ and dimension $N\geq3$ in the form
$$
u(x,t)=f(|x|t^{-1/2}), \quad f(\xi)\sim\left[-\frac{m-p}{m(N-2)}\ln\,\xi\right]^{1/(m-p)} \ {\rm as} \ \xi\to0.
$$ 
One may then expect that, similarly to the case $\sigma>-2$, this solution represents the asymptotic profile of general solutions to Eq. \eqref{eq1}, or at least provides the grow-up rate $(\ln\,t)^{1/(m-p)}$.

\medskip

The main difficulty in establishing grow-up rates in both critical cases (low dimensions and $\sigma=-2$) described at points 1 and 3 above, is the (apparent) lack of subsolutions with the expected behavior. Such subsolutions, if they exist, would clearly have to be constructed by coupling several pieces corresponding to the inner, transition and outer regions, as explained above. These problems will be addressed in a future work.

\bigskip

\noindent \textbf{Acknowledgements} The authors are partially supported by the Project PID2024-160967NB-I00 funded by AEI (Spain) and FEDER. M. L. is also partially supported by the project PID2022-136589NB-I00 funded by AEI (Spain) and FEDER.

\bigskip

\noindent \textbf{Data availability} Our manuscript has no associated data.

\bigskip

\noindent \textbf{Conflict of interest} The authors declare that there is no conflict of interest.

\bibliographystyle{plain}

\end{document}